\documentclass[a4paper,12pt]{amsart}
\makeindex

\usepackage{amsfonts,graphics,amsmath,amsthm,amsfonts,amscd,amssymb,amsmath,latexsym,euscript, enumerate, mathtools}
\usepackage{epsfig,url}
\usepackage{flafter}
\usepackage[all,cmtip,line]{xy}
\usepackage{array}
\usepackage[english]{babel}
\usepackage{overpic}
\usepackage{subfig}
\usepackage{multirow}
\usepackage{amsrefs}
\usepackage{tikz-cd}

\usepackage{wrapfig}
\usepackage[shortlabels]{enumitem}
\usepackage{graphicx}
\usepackage{microtype}
\usepackage{hyperref}

\newtheorem{theorem}{Theorem}[section]
\newtheorem{lemma}[theorem]{Lemma}
\newtheorem{proposition}[theorem]{Proposition}
\newtheorem{conjecture}[theorem]{Conjecture}

\newtheorem{corollary}[theorem]{Corollary}

\theoremstyle{definition} % italic or bold etc.
\newtheorem{definition}[theorem]{Definition}
\newtheorem{example}[theorem]{Example}

\newtheorem{remark}[theorem]{Remark}
\newtheorem{definition-theorem}[theorem]{Definition-Theorem}
\newtheorem{definition-lemma}[theorem]{Definition-Lemma}

\numberwithin{equation}{section}
\newcommand{\C}{\mathbb{C}}
\newcommand{\R}{\mathbb{R}}
\newcommand{\Z}{\mathbb{Z}}

\newcommand{\Q}{\mathbb{Q}}

\def\P{\mathbb{P}}

\def\div{\operatorname{div}}

\DeclareMathOperator{\Spec}{Spec}

\DeclareMathOperator{\Supp}{Supp}

\usepackage{mathtools}

\DeclarePairedDelimiterX{\norm}[1]{\lVert}{\rVert}{#1}

\newcommand{\ceil}[1]{\left\lceil #1 \right \rceil}
\let\oldframe\frame
\renewcommand\frame[1][allowframebreaks]{\oldframe[#1]}
\title[On volumes and the generic invariance of Fano type varieties]{On volumes and the generic invariance of Fano type varieties}

\begin{document}

\author[D.~Kim]{Donghyeon Kim}
\address[Donghyeon Kim]{Department of Mathematics, Yonsei University, 50 Yonsei-ro, Seodaemun-gu, Seoul 03722, Republic of Korea}
\email{narimial0@gmail.com}
\thanks{The author is grateful to the anonymous referee for their thoughtful comments and suggestions. D. Kim is partially supported by the Samsung Science and Technology Foundation under Project
Number SSTF-BA2302-03. }

\date{\today}

\begin{abstract}
We demonstrate the generic invariance of the Fano type property in the case where the volumes of anti-canonical divisors of Fano type fibers are a constant over a Zariski-dense subset, or the Fano type fibers are of dimension $2$. Additionally, paralleling this theorem, we establish a conjecture by Schwede and Smith under the condition that the volumes of anti-canonical divisors remain constant in the reduction mod $p$.
\end{abstract}

\maketitle
\allowdisplaybreaks

\section{Introduction}
The advancement of birational geometry and the Minimal Model Program (MMP) allows us to study the boundedness of varieties. Hacon, McKernan, and Xu proved a boundedness result for varieties with ample canonical divisors in \cite{HMX18}, while Birkar established in \cite{Bir19} and \cite{Bir24} that Fano type (fibration) maintains a good boundedness property. Here, a projective variety $X$ is of \emph{Fano type} if there is an effective $\Q$-Weil divisor $\Delta$ on $X$ with $(X,\Delta)$ being klt and $-(K_X+\Delta)$ ample.

\smallskip

A set of varieties $\mathcal{S}:=\{X_i\}_{i\in I}$ is \emph{bounded} if there exists a projective morphism $V\to W$ of varieties such that $X_i$ is a closed fiber of $V\to W$ for each $i\in I$. The next step after assessing boundedness is to inquire if $X_i$ has a good property; then so does the geometric generic fiber of $W$. For instance, if $X_i$ is Fano for each $i\in I$, the geometric generic fiber of $W$ is also Fano (cf. \cite{Laz04a}*{Theorem 1.2.17}).

\smallskip

Fano type varieties extend the concept of Fano in the context of MMP. If $X$ is Fano, $D$ a Cartier divisor on $X$, and $X\dashrightarrow X'$ a step in $D$-MMP, then generally $X'$ is not Fano. However, $X'$ is of Fano type. Thus, studying Fano type varieties follows naturally from the study of Fano varieties.

\smallskip

One may propose the following conjecture:

\begin{conjecture}\label{conj1}
Let $X\to S$ be a projective and surjective morphism between two normal varieties over $\C$, and let $S'\subseteq S$ be a Zariski-dense subset of $S$ such that for any closed point $s\in S'$, the fiber $X_s$ is of Fano type. Then the geometric generic fiber $X_{\overline{\eta}}$ is of Fano type, where $\eta\in S$ is the generic point.
\end{conjecture}

Note that the conjecture is equivalent to \cite{CLZ25}*{Question 1.1}. When $S'$ is an open subset of $S$, the conjecture holds by the semicontinuity property of cohomology, leaving us to deal with scenarios where $S'$ can be very sparse.

\smallskip

Let $X$ be a projective variety over an algebraically closed field $k$, and let $D$ be a Weil divisor on $X$. Then the volume of $D$ is defined as
$$ \mathrm{vol}(D):=\limsup_{m\to \infty}\frac{\dim_k H^0(X,\mathcal{O}_X(mD))}{\frac{m^{\dim X}}{(\dim X)!}}.$$

\smallskip

Our first result on the conjecture is as follows:

\begin{theorem}\label{imain1}
Suppose $(X,\Delta)$ is a projective pair over an algebraically closed field $k$ of characteristic $0$. Let $X\to S$ be a projective and surjective morphism between two normal varieties over $k$. Then the following are equivalent:

\begin{itemize}
    \item[\emph{(a)}] There exists a Zariski-dense subset $S' \subseteq S$ and a positive number $v>0$ satisfying the following conditions. For any closed point $s \in S'$:
\begin{itemize}
    \item[\emph{(1)}] The volume is given by $\mathrm{vol}(-(K_{X_s}+\Delta_s))=v$.
    \item[\emph{(2)}] The fiber $(X_s,\Delta_s)$ is of Fano type.
\end{itemize}
\item[\emph{(b)}] The geometric generic fiber $(X_{\overline{\eta}},\Delta_{\overline{\eta}})$ of $(X,\Delta)$ is of Fano type, where $\eta\in S$ is the generic point of $S$.
\end{itemize}
\end{theorem}

Note that if condition (1) in the theorem holds, then Conjecture \ref{conj1} follows immediately.

\smallskip

Moreover, we can prove Conjecture \ref{conj1} in the dimension $2$ case without any conditions on the volumes of anti-canonical divisors.

\begin{theorem} \label{surface}
Conjecture \ref{conj1} holds if the Fano type fibers have dimension $2$.
\end{theorem}

Note that the proof of Theorem \ref{surface} is almost the same as that of \cite{GT16}*{Theorem 1.1}. As a corollary of Theorems \ref{imain1} and \ref{surface}, we can prove the following.

\begin{corollary} \label{coro1}
Let $\mathcal{S}$ be a bounded family of Fano type surfaces over $\C$. Then the set
$$ \{\mathrm{vol}(-K_X)\mid X\in \mathcal{S}\}$$
is a DCC set.
\end{corollary}

Note that we can view Corollary \ref{coro1} as a partial anti-canonical version of the DCC property of canonical volumes as in \cite{HMX14}*{Theorem 1.3}. The boundedness condition is crucial in the proof. See Example \ref{ex}.

\smallskip

Next, we see the reduction mod $p$ side. In many aspects, the situations in reduction mod $p$ are parallel to the situation in characteristic $0$ side. For example, let $X$ be a $\Q$-Gorenstein variety over $\C$. If there is a finitely generated $\Z$-algebra $R$ such that $X_s$ is strongly $F$-regular for any $s\in \Spec R$, then $X$ is klt (cf. \cite{HW02}*{Theorem}). It can be viewed as a reduction mod $p$ counterpart of the fact that for a flat family $X\to S$, if $X_s$ is klt for any closed point $s\in S$, then the geometric generic point $X_{\overline{\eta}}$ is klt (cf. \cite{Laz04b}*{Theorem 9.5.16}) (for the definition of strong $F$-regularity, see Definition \ref{asdfaasf}).

\smallskip

Schwede and Smith proposed the following conjecture in \cite{SS10}:

\begin{conjecture}\label{conj2}
Let $X$ be a globally $F$-regular type variety over $\C$. Then $X$ is of Fano type.
\end{conjecture}

For the definition of the notion of globally $F$-regular type variety, refer to Definition \ref{type}. Note that global $F$-regularity can be considered as a ``right'' definition of the notion of Fano type over positive characteristic fields, and therefore we may view Conjecture \ref{conj2} as a reduction mod $p$ version of Conjecture \ref{conj1} (for the definition of global $F$-regularity, see Definition \ref{asdfaasf}). The conjecture can be used for Fano type results in many moduli spaces (cf. \cite{SZ20} and \cite{WW24}).

\smallskip

Conjecture \ref{conj2} is proved in the surface case in \cite{GT16}, \cite{HP15} and \cite{Oka17}. Moreover, the conjecture holds if $X$ is a Mori dream space by \cite{GOST15}. Furthermore, it is proved that if $X$ is a smooth globally $F$-regular type threefold with nef anti-canonical divisor, then the conjecture is true (\cite{CKT24}). As a parallel result of Theorem \ref{imain1}, we prove the following result.

\begin{theorem}\label{imain2}
Suppose $(X,\Delta)$ is a projective pair over $\C$. Let $R$ be a finitely generated $\Z$-algebra and $(X_R,\Delta_R)$ a model of $(X,\Delta)$ over $R$. Then the following are equivalent:

\begin{itemize}
    \item[\emph{(a)}] There exists a Zariski-dense subset $S'\subseteq \Spec R$ and a positive number $v>0$ satisfying the following conditions. For any closed point $s\in S'$:
\begin{itemize}
    \item[\emph{(1)}] The volume is given by $\mathrm{vol}(-(K_{X_{\overline{s}}}+\Delta_{\overline{s}}))=v$.
    \item[\emph{(2)}] The fiber $(X_{\overline{s}},\Delta_{\overline{s}})$ is globally $F$-regular.
    \end{itemize}
    \item[\emph{(b)}] The pair $(X,\Delta)$ is of Fano type.
\end{itemize}
(For the definition of the notion of ``model", see Definition \ref{model})
\end{theorem}

Let us explain two important technical steps in proving both Theorem \ref{imain1} and \ref{imain2}. First, we will use the log canonical threshold criterion for Fano type property (cf. Lemma \ref{Fano type}). For computing that, a priori, we must take into account all log resolutions of a pair. To address the problem, Xu's result on computing quasi-monomial valuations as documented in \cite{Xu20} is employed. Following this, a single log resolution suffices for our considerations. Second, we should compare log canonical thresholds, and thus Nakayama's asymptotic orders on both closed and generic fibers. To address the issue, we need to use Jiao's result on volumes of a family of big divisors (\cite{Jia25}). 

\smallskip

Let us note that the idea starts in \cite{BLX22} to prove the openness of $K$-semistability, while the part involving the comparison of log canonical thresholds in \cite{BLX22} is attributed to \cite{HMX13}.

\smallskip

This paper is organized as follows: In Section \ref{22}, we collect basic definitions, lemmas, and properties. In Section \ref{333}, we collect basic definitions and theorems in positive characteristic algebraic geometry. In Section \ref{444}, we introduce Jiao's result \cite{Jia25} and prove an analog of the result. Finally, in Section \ref{F}, we prove the main theorems.

\section{Preliminaries} \label{22}
In this paper, \emph{variety} means reduced, irreducible, separated, and finite type scheme over an algebraically closed field $k$. Any $\Q$-divisor is $\Q$-Weil unless otherwise stated. Let us collect basic notions and definitions which will be used in the paper.

\begin{itemize}
\item Given a variety $X$ and an effective $\Q$-divisor $\Delta$ on it, we say that $(X,\Delta)$ is a \emph{couple}. A couple $(X,\Delta)$ is a \emph{pair} when $K_X+\Delta$ is $\Q$-Cartier.
\item Consider $f:X'\to X$ as a birational morphism. The set where $f$ is not an isomorphism on $X'$ is denoted by $\mathrm{Exc}(f)$.
    \item Consider $X$ as a variety. A morphism $f: X'\to X$ is a \emph{resolution} if it is a proper birational morphism and $X'$ is smooth. For a pair $(X, \Delta)$, the morphism $f:X'\to X$ is a \emph{log resolution} of $(X,\Delta)$ if $f$ is a proper birational morphism, and the union $\mathrm{Exc}(f) \cup \mathrm{Supp}\, \Delta$ forms a simple normal crossing (snc) divisor on $X'$.
    \item For a variety $X$ and a Cartier divisor $D$ on $X$, let us denote by
    $$ |D|:=\{D'\sim D\mid D'\text{ is an effective Cartier divisor on }X\},$$
    and
    $$ |D|_{\Q}:=\{D'\sim_{\Q} D\mid D'\text{ is an effective }\Q\text{-Cartier }\Q\text{-divisor on }X\}.$$
    We denote by $\mathfrak{b}(|D|)$ the base locus of $|D|$.
    \item Let $S$ be a scheme, and let $s\in S$ be a point. We denote by $k(s)$ the residue field of $\mathcal{O}_{S,s}$. We may regard $s$ as a morphism $s:\Spec k(s)\to S$. Moreover, we denote by $\overline{s}$ the composition of $\Spec \overline{k(s)}\to \Spec k(s)$ and $\Spec k(s)\to S$. For an integral scheme $X$, we denote by $k(X)$ the function field of $X$.
    \item Let $X\to S$ be a morphism of schemes, and let $s\in S$ be a (not necessarily closed) point. Denote by $X_s$ the fiber product $X\times_S s$, and by $X_{\overline{s}}$ the fiber product $X\times_S \overline{s}$. We say that $X_{\overline{s}}$ is a \emph{geometric fiber} of $X\to S$.
    \item Let $(X,\Delta)$ be a pair, and let $E$ be a prime divisor over $X$. Suppose $f:X'\to X$ is a proper birational morphism from a normal variety $X'$ that has $E$ as an effective $f$-exceptional divisor. We denote by \emph{log discrepancy} $A_{X,\Delta}(E)$ of $(X,\Delta)$ along $E$, and the definition is
    $$ A_{X,\Delta}(E):=\mathrm{mult}_E(K_{X'}-f^*(K_X+\Delta))+1.$$
    Note that the definition does not depend on the choice of $f$.
\end{itemize}

The \emph{volume} of a divisor plays an important role in the paper.

\begin{definition}
Let $X$ be a proper finite type scheme over an algebraically closed field $k$, and let $D$ be a $\Q$-Weil divisor on $X$. The \emph{volume} of $D$ is defined by
$$ \mathrm{vol}(D):=\limsup_{m\to \infty}\frac{\dim_k H^0(X,\mathcal{O}(mnD))}{\frac{(mn)^{\dim X}}{(\dim X)!}},$$
where $n$ is a positive integer such that $nD$ is Weil.
\end{definition}

It is well-known that $\mathrm{vol}(D)<\infty$, and $\mathrm{vol}(D)>0$ if and only if $D$ is big. For more details, see \cite{Laz04a}*{Chapter 2.2} (in Cartier case) and \cite{FKL16} (in Weil case).

\subsection{Models}

Let us define the notion of \emph{models}.

\begin{definition} \label{model}
Consider that $X$ is a variety over $\C$.
\begin{itemize}
    \item[(a)] Let $Z_i\subseteq X$ be closed subschemes. Let $R$ be a finitely generated $\Z$-algebra with a flat $R$-scheme $X_R$ and closed subschemes $Z_{iR}\subseteq X_R$ that are flat over $R$, and $(Z_{iR}\subseteq X_R)\times_R \C=Z_i\subseteq X$ for each $i$. In this case, the tuple $(X_R,Z_{1R},\cdots,Z_{nR})$ is called a \emph{model} of $(X,Z_1,\cdots,Z_n)$ over $R$.
    \item[(b)] Let $D$ be an effective Weil divisor on $X$. We say that a flat $R$-scheme $D_R$ is a \emph{model} of $D$ over $R$ if $(X_R,D_R)$ is a model of $(X,D)$ over $R$.
    \item[(c)] Let $\left(X,\Delta:=\sum^r_{i=1}r_i\Delta_i\right)$ be a couple, where $r_i>0$ and $\Delta_i$ are prime divisors. We say $\left(X_R,\Delta_R:=\sum^r_{i=1}r_i\Delta_{iR}\right)$ is a \emph{model} of $(X,\Delta)$ over $R$ if $(X_R,\Delta_{1R},\cdots,\Delta_{rR})$ is a model of $(X,\Delta_1,\cdots,\Delta_r)$ over $R$.
    \item[(d)] Let $f:X'\to X$ be a morphism between varieties, $f_R:X'_R\to X_R$ morphism between flat $R$-schemes, and $f_R\times_R \C=f$. Then we say $f_R$ is a \emph{model} of $f$ over $R$.
    \item[(e)] Let $f:X\dashrightarrow X_0$ be a birational map with defining loci $U\subseteq X$, $U_0\subseteq X_0$, and the defining birational morphism $f_U:U\to U_0$. Let $X_R,X_{0R}$ and $f_{UR}:U_R\to U_{0R}$ be models of $X,X_0$ and $f_U$ over $R$. If $f_{UR}$ is birational and the complements of $U_R\subseteq X_R$ and $U_{0R}\subseteq X_{0R}$ are codimension $\le 1$, we say that $f_R:X_R\dashrightarrow X_{0R}$ is a \emph{model} of $f$ over $R$.
\end{itemize}
\end{definition}

\begin{remark} \label{remark}
Let $X\to S$ be a morphism between varieties, and let $f:X'\to X_{\overline{\eta}}$ be any morphism between varieties over $\overline{k(\eta)}$. Then $X'$ is determined by finitely many equations in $\overline{k(\eta)}$, and let $s_1=0,\cdots,s_r=0$ be the equations. There exists a finite extension $K'/k(\eta)$ such that all coefficients of $s_1,\cdots,s_r$ are in $K'$, and thus there exists a quasi-finite morphism $S'\to S$ such that all coefficients of $s_1,\cdots,s_r$ are in $\mathcal{O}_{S'}$. Therefore, there is an $S$-scheme $X'_S$ such that $X'_S\times_S \overline{\eta}=X'$. Moreover, it induces a morphism $f_S:X'_S\to X$ such that $f_S\times_S \overline{\eta}=f$.

\smallskip

Let $X$ be a variety over $\C$, and let $s_1=0,\cdots,s_r=0$ be the equations which determine $X$. Let $R$ be the finitely generated $\Z$-algebra generated by coefficients of $s_1,\cdots,s_r$. Then there exists a model $X_R$ of $X$ over $R$. Moreover, for any morphism $f:X'\to X$, there exists a finitely generated $\Z$-algebra $R$ and a model $f_R$ of $f$ over $R$.
\end{remark}

For a variety $S$, by ``shrinking $S$" we mean replacing $S$ with a Zariski open subset, and for a finitely generated $\Z$-algebra $R$, by ``shrinking $R$" we mean replacing $R$ with a finitely generated $R$-algebra.

\smallskip

\begin{lemma} \label{intersection}
Let $X$ be a projective variety over $\C$, $Z\subseteq X$ a closed subscheme of dimension $d$, $D_1,\cdots,D_d$ $\Q$-Cartier $\Q$-Weil divisors on $X$, and let $R$ be a finitely generated $\Z$-algebra. Given any model $(X_R,Z_R,D_{1R},\cdots,D_{dR})$ of $(X,Z,D_1,\cdots,D_d)$ over $R$, we have
$$ D_1 \cdots D_d\cdot Z=D_{1\overline{s}}\cdots D_{d\overline{s}}\cdot Z_{\overline{s}}$$
for a general closed point $s\in \Spec R$.
\end{lemma}

\begin{proof}
If we consider \cite{CKT24}*{Lemma 2.2}, the only remaining part is to show
$$ D_{1s}\cdots D_{ds}\cdot Z_{s}=D_{1\overline{s}}\cdots D_{d\overline{s}}\cdot Z_{\overline{s}}. $$
This equation results from \cite{FGAexplained}*{Definition B.8} alongside the flat base change theorem (see \cite{Stacks}*{Lemma 02KH}).
\end{proof}

%The notion of \emph{diminished base locus} is introduced in \cite{ELMNP06}, and we recall the definition.

%\begin{definition} \label{diminished}
%Let $X$ be a smooth projective variety, and let $D$ be a $\Q$-Cartier divisor on $X$. Denote by $\mathrm{Bs}(|D|)$ the base locus of $|D|$ if $D$ is Cartier.
%\begin{itemize}
%    \item[(a)] The \emph{stable base locus} $\B(D)$ is
%    $$ \B(D):=\bigcap_{m\ge 1} \Bs(|mnD|)_{\mathrm{red}},$$
%    where $n$ is a positive integer such that $nD$ is Cartier and $H^0(X,nD)\ne 0$.
%    \item[(b)] The \emph{diminished base locus} $\B_-(D)$ is
%    $$ \B_-(D):=\bigcup_{\ell\ge 1}\B\left(D+\frac{1}{\ell}A\right),$$
%    where $A$ is an ample Cartier divisor on $X$. The definition does not depend on the choice of $A$.
%\end{itemize}
%\end{definition}

\subsection{Log discrepancies in a family}
The subsection is devoted to the proof of the following two propositions. Let $k$ be an algebraically closed field of characteristic $0$.

\begin{proposition} \label{disc 1}
Let $(X,\Delta)$ be a pair over $k$, $f:X\to S$ a proper morphism between two varieties over $k$, $X'\to X$ a proper birational morphism, and let $E$ be a prime divisor on $X'$ which is flat and geometrically irreducible over $S$. Suppose $\eta\in S$ is the generic point. Then,
$$ A_{X_{\overline{\eta}},\Delta_{\overline{\eta}}}(E_{\overline{\eta}})=A_{X_s,\Delta_s}(E_s)$$
for a general closed point $s\in S$.
\end{proposition}

\begin{proof}
Note that
$$
\begin{aligned}
A_{X,\Delta}(E)&=\mathrm{mult}_E(f^*(K_X+\Delta)-K_{X'})+1
\\ &=\mathrm{mult}_{E_{\overline{\eta}}}(f^*_{\overline{\eta}}(K_{X_{\overline{\eta}}}+\Delta_{\overline{\eta}})-K_{X'_{\overline{\eta}}})+1
\\ &=A_{X_{\overline{\eta}},\Delta_{\overline{\eta}}}(E_{\overline{\eta}}),
\end{aligned}$$
and similarly, $A_{X,\Delta}(E)=A_{X_s,\Delta_s}(E_s)$ for a general $s\in S$.
\end{proof}

\begin{proposition} \label{disc 2}
Let $(X,\Delta)$ be a pair over $\C$, $f:X'\to X$ a proper birational morphism with $E$ a prime divisor on $X'$, $R$ a finitely generated $\Z$-algebra and $(X_R,\Delta_R)$ and $(X'_R,E_R)$ models of $(X,\Delta)$ and $(X',E)$ over $R$ respectively. Then after shrinking $R$,
$$ A_{X,\Delta}(E)=A_{X_{\overline{s}},\Delta_{\overline{s}}}(E_{\overline{s}})$$
for a general closed point $s\in \Spec R$.
\end{proposition}

\begin{proof}
Let $f':X''\to X'$ be a resolution such that $\mathrm{Exc}(f\circ f')\cup \Supp ((f\circ f')^{-1}_*\Delta)$ is snc. After shrinking $R$, we may assume that there exists a model $f'_R$ of $f'$ over $R$ (cf. Remark \ref{remark}). As in Proposition \ref{disc 1},
$$ A_{X,\Delta}(E)=A_{X_s,\Delta_s}(E_s)=A_{X_{\overline{s}},\Delta_{\overline{s}}}(E_{\overline{s}}).$$
\end{proof}

\subsection{Valuations}
In this subsection, we introduce the concept of quasi-monomial valuation for both the characteristic zero as well as those with positive characteristic. For further information, refer to \cite{JM12} and \cite{Can20}. Let $k$ denote an algebraically closed field of any characteristic.

\smallskip

Consider $X$ as a variety over $k$. A \emph{valuation} on $k(X)$ is defined as a function $\nu:k(X)\to \R$ with the property that,\begin{itemize}
    \item[(a)] $\nu(a)=0$ for all $a\in k$,
    \item[(b)] $\nu(fg)=\nu(f)+\nu(g)$ for all $f,g\in k(X)$, and
    \item[(c)] $\nu(f+g)\ge \min\{\nu(f),\nu(g)\}$.
\end{itemize} 
We set $\nu(0):=\infty$. Subsequently, we define $$\mathcal{O}_{\nu}:=\{f\in k(X)\mid\nu(f)\ge 0\}.$$ Furthermore, the \emph{center} of $\nu$ is characterized by a scheme-theoretic point $x=c_X(\nu)\in X$, which ensures the existence of a local inclusion $\mathcal{O}_{X,x}\hookrightarrow \mathcal{O}_{\nu}$ of local rings.

\smallskip

We define $$ \nu(\mathfrak{a}):=\min\left\{\nu(f)\mid f\in \mathfrak{a}\cdot \mathcal{O}_{X,c_X(\nu)}\right\}.$$ Using a valuation $\nu\in \mathrm{Val}_X$ and a nonzero ideal $\mathfrak{a}$ in $\mathcal{O}_X$, we can define a topology on $\mathrm{Val}_X$ as the weakest topology such that $\nu\mapsto \nu(\mathfrak{a})$ is continuous on $\mathrm{Val}_X$ for every nonzero ideal $\mathfrak{a}$ in $\mathcal{O}_X$.

\smallskip

We say $(X',E)$ is a \textit{log smooth model} over $X$ with respect to the resolution $f:X'\to X$ and snc divisor $E=\sum^r_{i=1}E_i$ on $X'$ if $f$ is an isomorphism outside $E$'s support. Let $\eta$ denote the generic point of an irreducible component of the intersection of some of the divisors $E_i$. Let $E_{i_1},\cdots,E_{i_k}$ be the prime divisors of $E$ such that the intersection is $\overline{\{\eta\}}$, and let $z_{i_1},\cdots,z_{i_k}$ represent local equations of $E_{i_1},\cdots,E_{i_k}$.

\smallskip

By the Cohen structure theorem, there exists an isomorphism $\widehat{\mathcal{O}_{X',\eta}}\cong k(\eta)[[z_{i_1},\cdots,z_{i_k}]]$. So, for any $s\in \mathcal{O}_{X',\eta}$, we can write $s$ as
$$ s=\sum_{(\beta_1,\cdots,\beta_k)\in \Z^k_{\ge 0}}a_{(\beta_1,\cdots,\beta_k)}z^{\beta_1}_{i_1}\cdots z^{\beta_k}_{i_k}.$$
for some $a_{(\beta_1,\cdots,\beta_k)}\in k$. For any $\alpha:=(\alpha_1,\cdots,\alpha_k)\in \R^k_{\ge 0}$, we define
$$
\nu_{(X',E),\alpha}(s):=\min\left\{\sum^k_{i=1}\alpha_i\beta_i\;\middle|\; a_{(\beta_1,\cdots,\beta_k)}\ne 0\right\}.
$$
Then it defines a valuation on $X$.

\smallskip

Moreover, we define
$$
A_{X,\Delta}(\nu_{(X',E),\alpha}):=\sum^k_{i=1}\alpha_iA_{X,\Delta}(E_i).
$$
If $(X'',E':=\sum^{r'}_{i=1}E'_i)$ is a log smooth model over $X$, $\eta'\in Z':=\bigcap^{r'}_{i=1}E'_i$ is the generic point of an irreducible component of $Z'$, $E'_{i'_1},\cdots,E'_{i'_{k'}}$ are prime divisors of $E'$ such that the intersection is $\overline{\{\eta'\}}$, and if $\alpha'\in \R^{k'}_{\ge 0}$ is a tuple such that
$$ \nu_{(X',E),\alpha}=\nu_{(X'',E'),\alpha'},$$
then $A_{X,\Delta}(\nu_{(X',E),\alpha})=A_{X,\Delta}(\nu_{(X'',E'),\alpha'})$ by \cite{JM12}*{Proposition 5.1} and \cite{Can20}*{Lemma 4.3}.

\smallskip

In this paragraph, we assume the base field of $X$ is characteristic $0$. For a log smooth model $\left(X,\sum^r_{i=1}E_i\right)$ with $\bigcap^r_{i=1} E_i$ integral, we define a \emph{restriction map} $\rho_{(X',E)}:\mathrm{Val}_X\to \mathrm{QM}(X',E)$ by
$$ \rho_{(X',E)}(\nu):=\nu_{(X',E),(\nu(E_1),\cdots,\nu(E_r))},$$
and define
$$ A_{X,\Delta}(\nu):=\sup_{(X',E)\text{ log smooth model}}A_{X,\Delta}(\rho_{(X',E)}(\nu)).$$

\smallskip

We define
$$ \mathrm{QM}_{\eta}(X',E):=\left\{\nu_{(X',E),\alpha}\;\middle|\;\alpha\in \R^k_{\ge 0}\right\},$$
and
$$ \mathrm{QM}(X',E):=\bigcup_{\eta}\mathrm{QM}_{\eta}(X',E)\subseteq \mathrm{Val}_X,$$
where $\eta$ runs through the generic points of irreducible components of $\bigcap^r_{i=1} E_i$.

\smallskip

Let $\mathfrak{a}_{\bullet}=\left\{\mathfrak{a}_m\right\}_{m\ge \Z_{\ge 0}}$ be a graded sequence of ideals in $\mathcal{O}_X$. Define
$$ \nu(\mathfrak{a}_{\bullet}):=\inf_{m\ge 1}\frac{\nu(\mathfrak{a}_m)}{m}.$$

\begin{lemma}[{cf. \cite{JM12}*{Lemma 6.1 and Corollary 6.4}}] \label{yae'}
Let $X$ be a variety over $k$, and let $\mathfrak{a}_{\bullet}$ be a graded sequence of ideals in $\mathcal{O}_X$. Then the function $\nu\mapsto \nu(\mathfrak{a}_{\bullet})$ is upper-semicontinuous on $\mathrm{Val}_X$. Moreover, if $k$ is characteristic $0$, then the function is continuous on $\mathrm{Val}_X\cap \{A_{X,\Delta}(\nu)<\infty\}$.
\end{lemma}

\begin{proof}
Note that for any positive integer $m$, the function $\nu\mapsto \frac{\nu(\mathfrak{a}_m)}{m}$ is continuous on $\mathrm{Val}_X$. Taking $\mathrm{inf}_{m\ge 1}$ preserves upper-semicontinuity, and we obtain the first assertion. For the second assertion, we may take a log resolution $f:X'\to X$ of $(X,\Delta)$, and use \cite{JM12}*{Proposition 6.4}.
\end{proof}

The notion of \emph{log canonical threshold} plays an important role in the theory of valuation.

\begin{definition}
Let $(X,\Delta)$ be a pair over $k$, and let $\mathfrak{a}_{\bullet}$ be a graded sequence of ideals in $\mathcal{O}_X$. We define
$$ \mathrm{lct}(X,\Delta,\mathfrak{a}_{\bullet}):=\inf_{\nu\in \mathrm{Val}^*_X}\frac{A_{X,\Delta}(\nu)}{\nu(\mathfrak{a}_{\bullet})},$$
where $\nu$ runs through the divisorial valuations over $X$ on both definitions.
\end{definition}

\begin{definition}
Let $(X,\Delta)$ be a pair over $k$, and let $\mathfrak{a}_{\bullet}$ be a graded sequence of ideals in $\mathcal{O}_X$. We say that $\nu_0\in \mathrm{Val}^*_X$ \emph{computes} $\mathrm{lct}(X,\Delta,\mathfrak{a}_{\bullet})$ if
$$ \mathrm{lct}(X,\Delta,\mathfrak{a}_{\bullet})=\frac{A_{X,\Delta}(\nu_0)}{\nu_0(\mathfrak{a}_{\bullet})}.$$
\end{definition}

Conjectures are presented in \cite{JM12}*{Conjecture 7.4}, and a weak version is partially proved by Xu in \cite{Xu20}.

\begin{theorem}[{cf. \cite{Xu20}*{Theorem 1.1}}] \label{Xu}
Let $(X,\Delta)$ be a klt pair over $k$ of characteristic $0$, and let $\mathfrak{a}_{\bullet}$ be a graded sequence of ideals in $\mathcal{O}_X$. Then there is a quasi-monomial valuation $\nu_0\in \mathrm{Val}^*_X$ computing $\mathrm{lct}(X,\Delta,\mathfrak{a}_{\bullet})$.
\end{theorem}

\subsection{Nakayama's asymptotic order}
In this subsection, we aim to define Nakayama's \emph{asymptotic order}. For more details, see \cite{Nak04}*{Chapter 3} and \cite{FKL16}.

\smallskip

Let $X$ be a normal projective variety over an algebraically closed field $k$ with arbitrary characteristic, and $\nu\in \mathrm{Val}_X$ be a valuation. Suppose $D$ is a big $\Q$-Cartier $\Q$-divisor on $X$, then define $$ \nu(|D|):=\inf\left\{\nu(D')\mid D'\in |D|\right\},$$ and the \emph{asymptotic order} $$ \sigma_{\nu}(D):=\inf\left\{\nu(D')\mid D'\in |D|_{\Q}\right\}.$$
Note that for any proper birational morphism $f:X'\to X$ from a normal variety $X'$, there is a one-to-one correspondence $D'\mapsto f^*D'$ between $|D|_{\Q}$ and $|f^*D|_{\Q}$, and $$\nu(D')=\nu(f^*D')$$ for any $D'\in |D|_{\Q}$. Thus, we have $\sigma_{\nu}(D)=\sigma_{\nu}(f^*D)$. If $\nu=\mathrm{ord}_E$ for some prime divisor $E$ over $X$, denote it by $\sigma_E(D):=\sigma_{\mathrm{ord}_E}(D)$.

\smallskip

Let us prove a lemma that is instrumental in the proofs of the main theorems.

\begin{lemma}[{cf. \cite{FKL16}*{Proposition 2.1}}] \label{volume asymptotic order}
Let $X$ be a normal projective variety over $k$, $D$ a big $\Q$-Cartier $\Q$-divisor, and let $E$ be a prime divisor on $X$. Suppose $a\ge 0$ is a rational number. Then $a>\sigma_E(D)$ if and only if
$$ \mathrm{vol}(D-aE)<\mathrm{vol}(D).$$
\end{lemma}

\begin{proof}
Let us prove the ``if" side. Let us assume $a\le \sigma_E(D)$. Then for any $m$ such that $am$ is an integer, $am\le \mathrm{ord}_E(|mD|)$, and thus the map
$$ H^0(X,\mathcal{O}_X(mD-amE))\hookrightarrow H^0(X,\mathcal{O}_{X}(mD))$$
is an isomorphism. This gives the proof of the ``if" side. The proof of the ``only if" side is the same as \cite{FKL16}*{Proposition 2.1}.
\end{proof}

\begin{remark}
Lemma \ref{volume asymptotic order} holds even if $E$ is not $\Q$-Cartier.
\end{remark}

Let us prove the following corollary of Lemma \ref{yae'}.

\begin{corollary} \label{yae}
Let $X$ be a normal projective variety over $k$, and let $D$ be a big Cartier divisor on $X$. Then the function $\nu\mapsto \sigma_{\nu}(D)$ is upper-semicontinuous on $\mathrm{Val}_X$. Moreover, if $k$ is characteristic $0$, then the function is continuous on $\mathrm{Val}_X\cap \{A_{X,\Delta}(\nu)<\infty\}$.
\end{corollary}

\begin{proof}
Let $\mathfrak{a}_m:=\mathfrak{b}(|mnD|)$, where $n$ is a positive integer such that $H^0(X,nD)\ne 0$. Then $\mathfrak{a}_{\bullet}:=\left\{\mathfrak{a}_{m}\right\}_{m\in \Z_{\ge 0}}$ forms a graded sequence of ideals in $\mathcal{O}_X$, and $\sigma_{\nu}(D)=\frac{1}{n}\nu(\mathfrak{a}_{\bullet})$ for any valuation $\nu\in \mathrm{Val}_X$. Therefore, Lemma \ref{yae'} proves the assertion.
\end{proof}

Let us define the notion of \emph{log canonical threshold} for a projective klt pair $(X,\Delta)$ and a big $\Q$-Cartier divisor $D$ on $X$.

\begin{definition}
Let $(X,\Delta)$ be a projective klt pair over a field of characteristic $0$, and let $D$ be a big $\Q$-Cartier $\Q$-divisor on $X$. We define
$$ \mathrm{lct}_{\sigma}(X,\Delta,D):=\inf_{\nu\in \mathrm{Val}^*_X}\frac{A_{X,\Delta}(\nu)}{\sigma_{\nu}(D)}.$$
Let us say that a valuation $\nu_0\in \mathrm{Val}^*_X$ \emph{computes} $\mathrm{lct}_{\sigma}(X,\Delta,D)$ if
$$ \mathrm{lct}_{\sigma}(X,\Delta,D)=\frac{A_{X,\Delta}(\nu_0)}{\sigma_{\nu_0}(D)}.$$
\end{definition}

\begin{corollary} \label{Xuapp}
Let $(X,\Delta)$ be a projective klt pair over a field of characteristic $0$, and let $D$ be a big $\Q$-Cartier $\Q$-divisor on $X$. Then there is a quasi-monomial valuation $\nu_0\in \mathrm{Val}^*_X$ computing $\mathrm{lct}_{\sigma}(X,\Delta,D)$.
\end{corollary}

\begin{proof}
Let $n$ be a positive integer such that $H^0(X,nD)\ne 0$, and let $\mathfrak{a}_m:=\mathfrak{b}(|mnD|)$ for any positive integer $m$. Then $\mathfrak{a}_{\bullet}:=\{\mathfrak{a}_m\}_{m\in \Z_{\ge 0}}$ forms a graded sequence of ideals in $\mathcal{O}_X$, and $\sigma_{\nu}(D)=\frac{1}{n}\sigma_{\nu}(\mathfrak{a}_{\bullet})$ for any valuation $\nu\in \mathrm{Val}_X$. Now, apply Theorem \ref{Xu}.
\end{proof}

\subsection{Weighted blow-up} \label{weighted}
In this subsection, we define \emph{weighted blow-up}.

\smallskip

Let $\left(X',E=\sum^r_{i=1}E_i\right)$ be a couple of a smooth scheme $X'$ and an snc divisor $E\subseteq X'$ with an affine chart $\left(U_j:=\Spec A_j,\sum^r_{i=1}\mathrm{div}\,z_{ij}\right)$, and let $(\alpha_1,\cdots,\alpha_r)$ be a set of positive integers. Define by $I_{c,j}$ ($c$ is a positive integer) the ideal in $A_j$ generated by monomials $z^{c_1}_{1j}\cdots z^{c_r}_{rj}$ with
$$ \sum_i c_i\alpha_i\ge c.$$
Then we can define the blowup $\mathrm{Proj}\left(\bigoplus_{c\ge 1} I_{c,j}\right)\to U_j$ along $I_{c,j}$, and patching those gives us the \emph{weighted blow-up} $X_{(\alpha_1,\cdots,\alpha_r)}\to X'$. Let us denote by $\mathcal{I}_{(X',E),\alpha,c}$ the ideal sheaf in $\mathcal{O}_{X'}$ that is given by the patching of $I_{c,j}$.

\smallskip

Let us prove the following lemma, which tells us that any weighted blow-up is stable under taking closed fiber.

\begin{lemma}[{cf. \cite{BLX22}*{Proposition 4.1}}] \label{hel}
Let $\left(X',E=\sum^r_{i=1}E_i\right)$ be a couple of a regular scheme $X'$ and an snc divisor $E\subseteq X'$, $X'\to S$ a smooth morphism with $S$ regular such that each $\bigcap_{i\in I'} E_i$ (where $I'\subseteq \{1,\cdots,r\}$) is smooth over $S$. Then for each closed point $s\in S$,
\begin{itemize}
    \item[\emph{(a)}] $E_s\subseteq X'_s$ is an snc divisor,
    \item[\emph{(b)}] for any weight $\alpha:=(\alpha_1,\cdots,\alpha_r)$, $\mathcal{I}_{(X',E),\alpha,c}|_s=\mathcal{I}_{(X'_s,E_s),\alpha,c}$,
    \item[\emph{(c)}] any weighted blow-up of a variety is normal,
    \item[\emph{(d)}] the fiber of $X_{\alpha}\to X'$ at $s\in S$ is the weighted blow-up of $(X'_s,E_s)$ of weight $(\alpha_1,\cdots,\alpha_r)$, and
    \item[\emph{(e)}] if $X$ is a variety, $(X',E)$ a log smooth model over $X$ and $\mathrm{gcd}(\alpha_1,\cdots,\alpha_r)=1$, then $\nu_{(X',E),\alpha}$ is a divisorial valuation of the form $\mathrm{ord}_F$ for some prime divisor $F$ over $X'$, and $F$ can be extracted by the weighted blow up $X_{\alpha}\to X'$.
\end{itemize}
\end{lemma}

\begin{proof}
In this proof, we let $x\in X'$ be a closed point with $s=f(x)$.

\smallskip

Let us prove (a). Let $z_1,\cdots,z_r$ be local equations of $E_1,\cdots,E_r$ in $\mathcal{O}_{X',x}$. Since $z_1,\cdots,z_r$ is a regular sequence in $\mathcal{O}_{X',x}$, we obtain 
\begin{equation}\label{bra}
    \mathcal{O}_{X'_s,x}/(z_{1s},\cdots,z_{is})\cong \mathcal{O}_{X',x}/(z_{1},\cdots,z_i)\otimes_{\mathcal{O}_S}k(s),
\end{equation} 
where $z_{is}:=z_i\otimes_{\mathcal{O}_S} k(s)$ for each $i=1,\cdots,r$. We need to prove $z_{1s},\cdots,z_{rs}$ is a regular sequence in $\mathcal{O}_{X'_s,x}$. Note that $z_{1s}$ is a regular sequence in $\mathcal{O}_{X'_s,x}$, and consider the map
$$ 0\to \mathcal{O}_{X',x}/(z_1)\overset{\cdot z_2}{\to} \mathcal{O}_{X',x}/(z_1)\to \mathcal{O}_{X',x}/(z_1,z_2)\to 0.$$
Since $\bigcap_{j=1,\cdots,i}E_j$ is flat over $S$ and (\ref{bra}) holds, it induces
$$ 0\to \mathcal{O}_{X'_s,x}/(z_{1s})\overset{\cdot z_{2s}}{\to} \mathcal{O}_{X'_s,x}/(z_{1s})\to \mathcal{O}_{X'_s,x}/(z_{1s},z_{2s})\to 0,$$
and therefore $z_{1s},z_{2s}$ is a regular sequence in $\mathcal{O}_{X'_s,x}$. The proof of the fact that $z_{1s},\cdots,z_{is}$ is a regular sequence in $\mathcal{O}_{X'_s,x}$ for each $i=1,\cdots,r$ is the same.

\smallskip

Let us prove (b). We claim
\begin{equation}\label{hea}
\mathcal{O}_{X,x}/\left(\mathcal{I}_{(X',E),\alpha,c}\right)_x\text{ is a flat }\mathcal{O}_{S,s}\text{-algebra.}
\end{equation}
Let $t_1,\cdots,t_{r'}$ be a regular system of parameters of $\mathcal{O}_{S,s}$ (cf. \cite{Stacks}*{Definition 00KU, Lemma 00NQ}), and let $t_1=p=\mathrm{char}\,k(s)$ if the characteristics of $k(s)$ and $k(S)$ are different. Moreover, let $z'_1,\cdots,z'_{r''}$ be a regular sequence in $\mathcal{O}_{X',x}$ such that $z_1,\cdots,z_r,t_1,\cdots,t_{r'},z'_1,\cdots,z'_{r''}$ is a regular system of parameters in $\mathcal{O}_{X',x}$ generating the maximal ideal. Then, by the Cohen structure theorem (cf. \cite{Stacks}*{Theorem 032A}), we obtain that
$$ \widehat{\mathcal{O}_{X',x}}\cong k(s)[[t_1,\cdots,t_{r'},z_1,\cdots,z_r,z'_1,\cdots,z'_{r'}]]$$
if the characteristics of $k(s)$ and $k(S)$ are the same, and there is a Cohen ring $\Lambda$ with
$$ \widehat{\mathcal{O}_{X',x}}\cong \Lambda[[t_2,\cdots,t_{r'},z_1,\cdots,z_r,z'_1,\cdots,z'_{r'}]]$$
if the characteristics are different. Moreover, 
$$\left(\mathcal{I}_{(X',E),\alpha,c}\right)_x\cdot \widehat{\mathcal{O}_{X',x}}=\left\{\sum_{(c_1,\cdots,c_r)\in \Z^r_{\ge 0}} a_{c_1,\cdots,c_r}z^{c_1}_1\cdots z^{c_r}_r\,\middle|\,a_{c_1,\cdots,c_r}\in R\text{ and }\sum^r_{i=1} c_i \alpha_i\ge c\right\},$$
and
$$ \widehat{\mathcal{O}_{X
',x}}/\left(\left(\mathcal{I}_{(X',E),\alpha,c}\right)_x\cdot \widehat{\mathcal{O}_{X',x}}\right)=\left\{\sum_{(c_1,\cdots,c_r)\in \Z^r_{\ge 0}} a_{c_1,\cdots,c_r}z^{c_1}_1\cdots z^{c_r}_r\,\middle|\,a_{c_1,\cdots,c_r}\in R\text{ and }\sum^r_{i=1} c_i \alpha_i<c\right\}$$
where 
$$R=k(s)[[t_1,\cdots,t_{r'},z'_1,\cdots,z'_{r''}]]$$ in the equal characteristic case, and $$R=\Lambda[[t_2,\cdots,t_{r'},z'_1,\cdots,z'_{r''}]]$$
in the mixed characteristic case. In this situation, $t_1,\cdots,t_r$ (in the mixed characteristic case, $t_1=p$) is a regular sequence in 
$$\widehat{\mathcal{O}_{X
',x}}/\left(\left(\mathcal{I}_{(X',E),\alpha,c}\right)_x\cdot \widehat{\mathcal{O}_{X',x}}\right).$$ 
Moreover, we have that the natural morphism
$$ \widehat{\mathcal{O}_{X
',x}}/\left(\left(\mathcal{I}_{(X',E),\alpha,c}\right)_x\cdot \widehat{\mathcal{O}_{X',x}}\right)\to \left(\mathcal{O}_{X',x}/\left(\mathcal{I}_{(X',E),\alpha,c}\right)_x\right)\otimes_{\mathcal{O}_{X,x}} \widehat{\mathcal{O}_{X',x}} $$
is an isomorphism. Hence, by \cite{Stacks}*{Lemma 00LM}, we have that $t_1,\cdots,t_r$ is a regular sequence in $\mathcal{O}_{X',x}/\left(\mathcal{I}_{(X',E),\alpha,c}\right)_x$ (recall \cite{Stacks}*{Lemma 00MC} that any completion is faithfully flat), and applying \cite{Stacks}*{Lemma 07DY} gives us (\ref{hea}).

\smallskip

Therefore, $\mathrm{Tor}^{\mathcal{O}_{S,s}}_1\left(k(s),\mathcal{O}_{X',x}/\left(\mathcal{I}_{(X',E),\alpha,c}\right)_x\right)=0$. Consider an exact sequence
$$ 0\to \left(\mathcal{I}_{(X',E),\alpha,c}\right)_x\to \mathcal{O}_{X',x}\to \mathcal{O}_{X',x}/\left(\mathcal{I}_{(X',E),\alpha,c}\right)_x\to 0.$$
By taking $k(s)$ on the exact sequence, we obtain an injection
$$ 0\to \left(\mathcal{I}_{(X',E),\alpha,c}\right)_x\otimes_{\mathcal{O}_{S,s}}k(s)\to \mathcal{O}_{X'_s,x}.$$
Note that the image is $\left(\mathcal{I}_{(X'_s,E_s),\alpha,c}\right)_x$, and thus we proved (b).

\smallskip

Let us prove (c) for the variety case. Note that $\mathcal{I}_{(X',E),\alpha,c}$ is the integral closure of $\left(z^{\frac{c}{\alpha_1}}_1,\cdots,z^{\frac{c}{\alpha_r}}_r\right)$ for any sufficiently divisible $c$. Hence, by applying \cite{HS06}*{Proposition 5.2.1, Theorem 18.4.2}, we obtain the assertion.

\smallskip

(d) can be proved by the explicit description of the weighted blow-up and (b).

\smallskip

(e) is just \cite{LX20}*{Definition 2.8}.
\end{proof}

\section{Positive characteristic theory} \label{333}
Let $k$ be an algebraically closed field of characteristic $p>0$, and let $X$ be a finite type scheme over $k$. Then the \emph{Frobenius} $F_X:X\to X$ is defined by the identity on the underlying topological space of $X$, and $p$-th power on the corresponding map of structure sheaves $\mathcal{O}_X\to (F_X)_*\mathcal{O}_X$. For a positive integer $e$, let us denote by $F^e_X:X\to X$ the $e$-th iteration of the Frobenius. If there is no confusion, we denote by $F$ the Frobenius on $X$.

In this section, all varieties and pairs are over $k$.

\subsection{Strong and global $F$-regularity}
The notion of strong $F$-regularity is introduced in \cite{HH89}, and later a global analog of the notion, global $F$-regularity is introduced in \cite{Smi00}. Furthermore, a relation of global $F$-regularity and Fano type property is established in \cite{SS10}. In this subsection, we collect basic definitions and properties on the subject.

\smallskip

Let $X$ be a normal variety, and fix any effective divisor $D$ on $X$. We have an inclusion $\mathcal{O}_X\hookrightarrow \mathcal{O}_X(D)$. Thus, we also have an inclusion $F^e_*\mathcal{O}_X\hookrightarrow F^e_*\mathcal{O}_X(D)$, and thus the composition
$$ \mathcal{O}_X\hookrightarrow F^e_*\mathcal{O}_X\hookrightarrow F^e_*\mathcal{O}_X(D).$$
The composition provides the definition of \emph{globally $F$-regular variety}.

\begin{definition} \label{asdfaasf}
Let $(X,\Delta)$ be a pair. Then $(X,\Delta)$ is \emph{globally $F$-regular} if, for every effective divisor $D$, there exists some $e>0$ such that the composition $\mathcal{O}_X\hookrightarrow F^e_*\mathcal{O}_X(\ceil{(p^e-1)\Delta}+D)$ splits in the category of $\mathcal{O}_X$-modules.

Moreover, $(X,\Delta)$ is \emph{strongly $F$-regular} if there is an open covering $\{U_i\}$ of $X$ such that $(U_i,\Delta|_{U_i})$ is globally $F$-regular.
\end{definition}

Recall one of the main theorems in \cite{SS10}.

\begin{theorem}[{cf. \cite{SS10}*{Theorem 1.1, Remark 4.11}}] \label{x}
Let $(X,\Delta)$ be a projective globally $F$-regular pair. Then there is an effective $\Q$-divisor $\Delta'$ such that
\begin{itemize}
    \item[\emph{(a)}] $(X,\Delta+\Delta')$ is globally $F$-regular, and
    \item[\emph{(b)}] $K_X+\Delta+\Delta'\sim_{\Q} 0$.
\end{itemize}
\end{theorem}

Finally, let us recall the definition of \emph{globally $F$-regular type pair}.

\begin{definition} \label{type}
Let $(X,\Delta)$ be a pair over $\C$. We say that $(X,\Delta)$ is \emph{globally $F$-regular type} if there is a finitely generated $\Z$-algebra $R$ with a model $(X_R,\Delta_R)$ of $(X,\Delta)$ over $R$, and a Zariski-dense subset $S'\subseteq \Spec R$ such that for any $s\in S'$, the geometric closed fiber $(X_{\overline{s}},\Delta_{\overline{s}})$ is globally $F$-regular.
\end{definition}

\subsection{Test ideal}
Roughly, test ideal is an analog of the multiplier ideal sheaf in positive characteristic fields. There are many candidates in defining the test ideal, and we recall the definition of \emph{big test ideal}.

\begin{definition}[{cf. \cite{BSTZ10}*{Definition-Proposition 3.3}}]
Suppose that $(X=\Spec R,\Delta,\mathfrak{a}^t)$ is a triple where $R$ is an $F$-finite normal domain of characteristic $p>0$. Then the (big) \emph{test ideal} $\tau(X,\Delta,\mathfrak{a}^{\lambda})$ is the sum
$$ \sum_{e\ge 0}\sum_{\phi}\phi(F^e_*(d\mathfrak{a}^{\ceil{\lambda(p^e-1)}})),$$
where $\phi$ ranges over $\phi\in \mathrm{Hom}_R\left(F^e_*R(\ceil{(p^e-1)\Delta}),R\right)\subseteq \mathrm{Hom}_R(F^e_*R,R)$ and where $d$ is a big sharp test element for $(X,\Delta,\mathfrak{a}^{\lambda})$ (for the definition of \emph{big sharp test element}, see \cite{Sch10}*{Definition 2.16}).

Moreover, let $(X,\Delta,\mathfrak{a}^{\lambda})$ be a triple in general where $X$ is an $F$-finite scheme. Then the \emph{test ideal} $\tau(X,\Delta,\mathfrak{a}^{\lambda})$ is defined by an ideal sheaf of $\mathcal{O}_X$ such that for any affine open $U\subseteq X$, $\tau(X,\Delta,\mathfrak{a}^{\lambda})|_U=\tau(U,\Delta|_U,\left(\mathfrak{a}|_U\right)^{\lambda})$.

If $X$ is a normal projective variety and $D$ a Cartier divisor, then we denote by
$$ \tau(X,\Delta,\lambda|D|):=\tau(X,\Delta,\mathfrak{b}(|D|)^{\lambda})$$
for any $\lambda\ge 0$. We denote by $\tau(X,\Delta):=\tau(X,\Delta,\mathcal{O}_X)$.
\end{definition}

\begin{remark}
The notion of a big test ideal coincides with the usual test ideal for any pair $(X,\Delta)$ (cf. \cite{BSTZ10}*{Proposition 3.7}). Moreover, the big test ideal behaves well in localization, and so we recall the definition of the big test ideal instead of the usual test ideal.
\end{remark}

We have a simple criterion of strong $F$-regularity using the test ideal.

\begin{lemma}[{cf. \cite{Tak04}*{Corollary 2.10}}] \label{cri}
Let $(X,\Delta)$ be a pair. Then $(X,\Delta)$ is strongly $F$-regular if and only if $\tau(X,\Delta)=\mathcal{O}_X$.
\end{lemma}

The following theorem bridges positive characteristic theory and the theory of singularities in the minimal model program.

\begin{theorem} \label{test to multiplier}
Let $(X,\Delta)$ be a pair, $\mathfrak{a}$ an ideal in $\mathcal{O}_X$, $f:X'\to X$ a proper birational morphism from a normal variety $X'$ such that $\mathfrak{a}\cdot \mathcal{O}_{X'}=\mathcal{O}_{X'}(-F)$ for some effective divisor $F$ on $X'$. Then
$$ \tau(X,\Delta,\mathfrak{a}^t)\subseteq f_*\mathcal{O}_{X'}(\ceil{K_{X'}-f^*(K_X+\Delta)-tF})$$
\end{theorem}

\begin{proof}
Let $m$ be a positive integer and $s\in \mathfrak{a}^{\ceil{m\lambda}}$. It follows from \cite{Tak04}*{Theorem 2.13} that
$$
\begin{aligned}
\tau\left(X,\Delta+\frac{1}{m}\div s\right)&\subseteq f_*\mathcal{O}_{X'}\left(\ceil{K_{X'}-f^*\left(K_X+\Delta+\frac{1}{m}\div s\right)}\right)
\\ &\subseteq f_*\mathcal{O}_{X'}\left(\ceil{K_{X'}-f^*\left(K_X+\Delta\right)-\lambda F}\right).
\end{aligned}
$$
\end{proof}

\section{Fujita approximation and Jiao's result on volumes of big divisors} \label{444}
In this section, we state the main theorem of \cite{Jia25}, and prove an analog of the main theorem. Let us recall the Fujita approximation theorem.

\begin{theorem}[{cf. \cite{Fuj94}*{Theorem} and \cite{Tak07}*{Theorem 0.1}}] \label{Fujita}
Let $X$ be a projective variety, and $D$ a big Cartier divisor on $X$. Then, for an arbitrarily small real number $\varepsilon>0$, there exists a birational morphism $f:X'\to X$ of projective varieties and a decomposition
$$ f^*D=A+E$$
which satisfies the following conditions:
\begin{itemize}
    \item[(a)] $A$ is an ample divisor $E$ is an effective divisor on $X$.
    \item[(b)] $\mathrm{vol}(A)>\mathrm{vol}(D)-\varepsilon$.
\end{itemize}
\end{theorem}

Note that the characteristic of $k$ could be positive. Next, let us recall the main result of \cite{Jia25}.

\begin{theorem}[{cf. \cite{Jia25}*{Theorem 1.1}}] \label{main111}
Let $f:X\to S$ be a projective flat morphism of varieties over a field of characteristic $0$ and $D$ a Cartier divisor on $X$. Suppose $X_{\overline{\eta}}$ is the geometric generic fiber of $f$. Write $D_{\overline{\eta}}:=D|_{X_{\overline{\eta}}}$ and $D_s:=D|_{X_s}$ for every closed point $s\in S$. Then
$$ \mathrm{vol}(D_{\overline{\eta}})=\inf_{s\in S'}\mathrm{vol}(D_s)$$
for any Zariski-dense subset $S'\subseteq S$.
\end{theorem}

The following theorem can be considered as a reduction mod $p$ version of Theorem \ref{main111}.

\begin{theorem} \label{main222}
Let $X$ be a normal projective variety over $\C$, $D$ a Cartier divisor, and let $R$ be a finitely generated $\Z$-algebra such that there is a model $X_R$ of $X$ over $R$. For a closed point $s\in \Spec R$,
$$ \mathrm{vol}(D)=\inf_{s\in S'}\mathrm{vol}(D_{\overline{s}})$$
for any Zariski-dense subset $S'\subseteq \Spec R$.
\end{theorem}

\begin{proof}
Note that the proof is almost the same as \cite{Jia25}*{Proof of Theorem 1.1}.

\smallskip

Using the upper semicontinuity property of cohomology (refer to \cite{Har77}*{Theorem 12.8}), it follows that
$$ \mathrm{vol}(D) \leq \inf_{s \in S'} \mathrm{vol}(D_s). $$
Assume this result holds when $D$ is big. If $D$ is not big $D$, we choose an ample divisor $H_R$ on $X_R$. We define $H := H_R \times_R \C$ and
$$ t_0 := \inf\{ t \in \R_{\ge 0} \mid D + tH \text{ becomes big}\}. $$
By applying the given result to $D + tH$ for each $t > t_0$, it results in
$$ \mathrm{vol}(D+tH) = \inf_{s \in S'} \mathrm{vol}(D_{\overline{s}}+tH_{\overline{s}}). $$
Let us examine these two scenarios:

\smallskip

\noindent\textbf{Case 1}. If $D$ is pseudo-effective, then $t_0=0$. By letting $t\to 0^+$, we have
$$ \mathrm{vol}(D)=\inf_{s\in S'}\mathrm{vol}(D_{\overline{s}})=0.$$

\smallskip

\noindent\textbf{Case 2}. If $D$ is not pseudo-effective, then by \cite{GT16}*{Lemma 4.1}, for a general closed point $s\in \Spec R$, $D_{\overline{s}}$ is not pseudo-effective. Therefore, there is $s\in S'$ such that $D_{\overline{s}}$ is not pseudo-effective, and thus
$$ \mathrm{vol}(D)=\inf_{s\in S'}\mathrm{vol}(D_{\overline{s}})=0.$$
\smallskip

From now on we may assume $D$ is big, and we fix a sufficiently ample divisor $H_R$ on $X_R$ so that both $H_R \pm D_R$ are ample on $X_R$. Let us define $H:=H_R\times_R \C$. Consider any ring extension $R\subseteq R'$ where $R'$ is a finitely generated $\Z$-algebra. For any closed point $s'\in \Spec R'$ with its image $s$ in $\Spec R$, it holds that $X_{\overline{s}}=X_{\overline{s'}}$. Therefore, if required, we may replace $R$ with $R'$.

\smallskip

Let $\varepsilon>0$ and consider $f:X'\to X$, which serves as a Fujita approximation for $D$ (see Theorem \ref{Fujita}). This approximation satisfies $f^*D\sim_{\Q}A+E$, where $X'$ is smooth, $E\ge 0$, and $A$ is ample. Further, it holds that 
$$ \mathrm{vol}(A)\ge \mathrm{vol}(D)-\varepsilon.$$
By shrinking $R$, we can assume the existence of models $f_R:X'_R\to X_R$, $D_R$, $A_R$, and $E_R$ corresponding to $f:X'\to X$, $D$, $A$, and $E$ over $R$, respectively (refer to Remark \ref{remark}). Additionally, we can assume that
$$ f^*_RD_R\sim_{\Q}A_R+E_R.$$

\smallskip

Define $d$ as $\dim X$. According to \cite{BPDD13}*{Theorem 4.1} or \cite{Laz04b}*{Theorem 11.4.21}, there is a universal constant $C>0$ such that
$$ (A^{d-1}\cdot E)^2 \leq C \cdot (H^d) \cdot (\mathrm{vol}(D) - \mathrm{vol}(A)) \leq C\varepsilon \cdot (H^d),$$
and define $C':=\sqrt{C \cdot (H^d)}$. Note that $C'$ depends only on $H$ and $A^{d-1} \cdot E$.

\smallskip

Observe that
$$ A^{d-1}\cdot (A+E)\le A^d+C'\sqrt{\varepsilon}.$$
Utilizing Lemma \ref{intersection}, it follows that
\begin{equation}\label{1} A^{d-1}_{\overline{s}}\cdot (A_{\overline{s}}+E_{\overline{s}})\le A^d_{\overline{s}}+C'\sqrt{\varepsilon}
\end{equation}
for every general closed point $s\in \Spec R$.

\smallskip

Note $A_{\overline{s}}+E_{\overline{s}}\sim_{\Q} f^*_{\overline{s}}D_{\overline{s}}$ for every general $s\in \Spec R$. Fix a closed point $s\in S'$ and choose $\varepsilon'>0$. By the Fujita approximation theorem (cf. Theorem \ref{Fujita}), we may choose a projective birational morphism $g:X''\to X'_{\overline{s}}$ such that
$$ g^*f^*_{\overline{s}}D_{\overline{s}}\sim_{\Q}g^*(A_{\overline{s}}+E_{\overline{s}})\sim_{\Q}A'_{\overline{s}}+E'_{\overline{s}},$$
where $E'_{\overline{s}}\ge 0$, $A'_{\overline{s}}$ is ample, and
$$ \mathrm{vol}(A'_{\overline{s}})\ge \mathrm{vol}(A_{\overline{s}}+E_{\overline{s}})-\varepsilon'=\mathrm{vol}(D_{\overline{s}})-\varepsilon'.$$
We have
\begin{equation} \label{2}(g^*A_{\overline{s}})^{d-1}\cdot A'_{\overline{s}}\le (g^*A_{\overline{s}})^{d-1}\cdot (A'_{\overline{s}}+E'_{\overline{s}})=A^{d-1}_{\overline{s}}\cdot (A_{\overline{s}}+E_{\overline{s}}).
\end{equation}
If we apply the generalized inequality of Hodge type (cf. \cite{Laz04a}*{Theorem 1.6.1 and Remark 1.6.5}), we have
\begin{equation}\label{3} (A^d_{\overline{s}})^{\frac{d-1}{d}}\left(A^{'d}_{\overline{s}}\right)^{\frac{1}{d}}\le (g^*A_{\overline{s}})^{d-1}\cdot A'_{\overline{s}}.
\end{equation}
Now, combine (\ref{1}), (\ref{2}) and (\ref{3}), we have
$$ (A^d_{\overline{s}})^{\frac{d-1}{d}}\left(A^{'d}_{\overline{s}}\right)^{\frac{1}{d}}\le (g^*A_{\overline{s}})^{d-1}\cdot A'_{\overline{s}}\le A^{d-1}_{\overline{s}}\cdot (A_{\overline{s}}+E_{\overline{s}})\le A^d_{\overline{s}}+C'\sqrt{\varepsilon}.$$
Thus,
$$ A^{'d}_{\overline{s}}\le \left((A^d_{\overline{s}})^{\frac{1}{d}}+\frac{C'\sqrt{\varepsilon}}{(A^d_{\overline{s}})^{\frac{d-1}{d}}}\right)^d.$$
By the Serre's vanishing theorem and Lemma \ref{intersection},
$$ \mathrm{vol}(A)=A^d=A^d_{\overline{s}}=\mathrm{vol}(A_{\overline{s}})\text{ for every general }s\in \Spec R.$$
Moreover,  $A^{'d}_{\overline{s}}\ge \mathrm{vol}(D_{\overline{s}})-\varepsilon'$, and thus
$$
\begin{aligned}
\mathrm{vol}(D_{\overline{s}})&\le A^{'d}_{\overline{s}}+\varepsilon'
\\ &\le \left((A^d_{\overline{s}})^{\frac{1}{d}}+\frac{C'\sqrt{\varepsilon}}{(A^d_{\overline{s}})^{\frac{d-1}{d}}}\right)^d+\varepsilon'
\\ &\le \left(\mathrm{vol}(A)^{\frac{1}{d}}+\frac{C'\sqrt{\varepsilon}}{\mathrm{vol}(A)^{\frac{d-1}{d}}}\right)^d+\varepsilon'
\\ &\le \left(\mathrm{vol}(D)^{\frac{1}{d}}+\frac{C'\sqrt{\varepsilon}}{(\mathrm{vol}(D)-\varepsilon)^{\frac{d-1}{d}}}\right)^d+\varepsilon'.
\end{aligned}
$$
Therefore, for any $\varepsilon''>0$, there exists $s\in S'$ such that
$$ \mathrm{vol}(D_{\overline{s}})\le \mathrm{vol}(D)+\varepsilon'',$$
and this is the theorem we want to prove.
\end{proof}

\section{Proof of the main theorems}\label{F}
The goal of this section is to provide the definition of Fano type couples, and prove Theorem \ref{imain1}, \ref{surface}, Corollary \ref{coro1} and Theorem \ref{imain2}. Let us define the following notion.

\begin{definition}
Let $(X,\Delta)$ be a projective couple. We say that $(X,\Delta)$ is \emph{of Fano type} if there is an effective $\Q$-divisor $\Delta'$ on $X$ such that $(X,\Delta+\Delta')$ is a klt pair and $-(K_X+\Delta+\Delta')$ is ample.
\end{definition}

\begin{lemma}[{cf. \cite{Xu23}*{Lemma 3.1}}] \label{Fano type}
Let $(X,\Delta)$ be a projective pair with big $-(K_X+\Delta)$. Then $(X,\Delta)$ is of Fano type if and only if $\mathrm{lct}_{\sigma}(X,\Delta,-(K_X+\Delta))>1$.
\end{lemma}

\begin{proof}
The proof of ``if" side is the same as the proof of \cite{Xu23}*{Lemma 3.1}, and thus we only treat with the ``only if" side.

\smallskip

Since $(X,\Delta)$ is of Fano type, there is an effective $\Q$-divisor $\Delta'$ such that $(X,\Delta+\Delta')$ is klt and $K_X+\Delta+\Delta'\sim_{\Q}0$. Moreover, there is a positive number $\varepsilon_0>0$ such that $(X,\Delta+(1+\varepsilon_0)\Delta')$ is klt. Thus, for any prime divisor $E$ over $X$, $A_{X,\Delta+(1+\varepsilon_0)\Delta'}(E)>0$, and
$$ \frac{A_{X,\Delta}(E)}{\mathrm{ord}_E\Delta'}\ge 1+\varepsilon_0.$$
Therefore
$$ \frac{A_{X,\Delta}(E)}{\sigma_E(-(K_X+\Delta))}\ge 1+\varepsilon_0,$$
and taking $\mathrm{inf}_E$ on both sides gives us the proof.
\end{proof}

\begin{lemma} \label{Fa1}
Let $(X,\Delta)$ be a pair of Fano type over a field of characteristic $0$, and let
$$ (X_0,\Delta_0):=(X,\Delta)\dashrightarrow (X_1,\Delta_1)\dashrightarrow \cdots \dashrightarrow (X_n,\Delta_n)$$
be a $-(K_X+\Delta)$-MMP. Then for any $i$, $(X_i,\Delta_i)$ is of Fano type.
\end{lemma}

\begin{proof}
Suppose $\Delta'$ is an effective $\Q$-Weil divisor on $X$ such that $(X,\Delta+\Delta')$ is klt and $K_X+\Delta+\Delta'$ is an anti-ample $\Q$-Cartier divisor. Moreover, there is an effective $\Q$-Weil divisor $\Delta''$ on $X$ with $(X,\Delta+\Delta'+\Delta'')$ klt and
$$ K_X+\Delta+\Delta'+\Delta''\sim_{\Q} 0.$$
By the contraction theorem (cf. \cite{KM98}*{Theorem 3.7 (4)}), $(X_i,\Delta_i+\Delta'_i+\Delta''_i)$ is klt and $K_{X_i}+\Delta_i+\Delta'_i+\Delta''_i\sim_{\Q} 0$. Moreover, $-(K_{X_i}+\Delta_i+\Delta'_i)$ is big for any $i$, and thus we obtain the assertion.
\end{proof}

\begin{lemma} \label{Fa2}
Let $(X,\Delta)$ be a pair of Fano type over a field of characteristic $0$, and let us assume that $-(K_X+\Delta)$ is a big and semiample $\Q$-Cartier $\Q$-divisor on $X$. Suppose that $f:X\to Y$ is the semiample fibration associated to $-(K_X+\Delta)$. Then $Y$ is of Fano type.
\end{lemma}

\begin{proof}
Note that $f$ is $(K_X+\Delta)$-trivial, therefore $(Y,\Delta_Y)$ is klt if we let $\Delta_Y:=f_*\Delta$. Moreover, $-(K_Y+\Delta_Y)$ is ample, and hence the assertion is proved.
\end{proof}

Let us prove the main theorems.

\begin{proof}[Proof of Theorem \ref{imain1}]
Let us prove (a)$\implies$(b); Note that $-(K_{X_{\overline{\eta}}}+\Delta_{\overline{\eta}})$ is big by Theorem \ref{main111}. Since $(X_s,\Delta_s)$ is klt for any $s\in S'$, $(X,\Delta)$ becomes klt after shrinking $S$ (cf. \cite{Laz04b}*{Example 9.5.6}), thus making $(X_{\overline{\eta}},\Delta_{\overline{\eta}})$ klt.

\smallskip

Using Theorem \ref{Xuapp}, there exists a log smooth model $(X',E)$ over $X_{\overline{\eta}}$ and a quasi-monomial valuation $\nu_0\in \mathrm{QM}\left(X',E=\sum E_i\right)$ that computes $\mathrm{lct}_{\sigma}(X_{\overline{\eta}},\Delta_{\overline{\eta}},-(K_{X_{\overline{\eta}}}+\Delta_{\overline{\eta}}))$. Let $E_1,\cdots,E_r\subseteq E$ be prime divisors that define $\nu_0$, and $\alpha_0\in \R^r_{\ge 0}$ a tuple satisfying $\nu_0=\nu_{(X',E),\alpha_0}$.

\smallskip

Since the morphism $X'\to X_{\overline{\eta}}$ and $E$ are determined by finitely many equations in $\overline{k(\eta)}$, after shrinking $S$ and replacing $S$ with a finite cover, we may assume that there is a log smooth model $(X'_S,E_S)$ over $X$ such that $(X'_S,E_S)\times_S \overline{\eta}=(X',E)$ (cf. Remark \ref{remark}). By generic flatness (cf. \cite{Stacks}*{Proposition 052A}), we may further assume that
\begin{itemize}
\item $S$ is regular,
    \item $E_S$ is flat over $S$ and has simple normal crossing,
    \item $X\to S$ and $X'_S\to S$ are flat with reduced and irreducible fibers, and
    \item $X'_S\to X$ is a fiberwise birational morphism over $S$.
\end{itemize}

\smallskip

If we consider Proposition \ref{disc 1}, then we have that after shrinking $S$,
\begin{equation} \label{main1}
A_{X_{\overline{\eta}},\Delta_{\overline{\eta}}}(\nu_{(X',E),\alpha_0})=A_{X_s,\Delta_s}(\nu_{(X'_s,E_s),\alpha_0})
\end{equation}
for any $s\in S$.

\smallskip

Let $\alpha:=(\alpha_1,\cdots,\alpha_r)\in \Q^r_{\ge 0}$, and let $n$ be the least common multiple of the denominators of $\alpha_1,\cdots,\alpha_r$. Take a weighted blow-up $X''_S\to X'_S$ of weight $(n\cdot \alpha_1,\cdots,n\cdot \alpha_r)$, $X''_s:=X''|_s$ for each $s\in S'$, and let $X'':=X''_S\times_S \overline{\eta}$. Note that $X''_s$ and $X''$ are weighted blowups of $(X'_s,E_s)$ and $(X',E)$ of weight $(n\cdot \alpha_1,\cdots,n\cdot \alpha_r)$ respectively (cf. Lemma \ref{hel} (d)).

\smallskip

The varieties and morphisms can be illustrated by the following diagram:
$$
\begin{tikzcd}
X''_{s}\ar[hook]{d}\ar["g_s"]{r}& X'_{s} \ar[hook]{d}\ar{r}& X_{s}\ar[hook]{d} \\
X''_S \ar{r}& X'_S \ar{r}& X \\
X'' \ar[hook]{u}\ar["g"]{r}& X' \ar[hook]{u}\ar{r}& X_{\overline{\eta}}\ar[hook]{u}
\end{tikzcd}
$$
The vertical morphisms are fiber products, and the horizontal morphisms are birational morphisms.
\smallskip

Let $s\in S'$, and let us note the facts we will use:

\begin{itemize}
    \item[(a)] $X''$ and $X''_s:=X''_S|_s$ are normal (cf. Lemma \ref{hel} (c)), and
    \item[(b)] if $F$ and $F_s$ are toroidal divisors extracted by $g$ and $g_s:=g_S|_s$, then
    \begin{equation} \label{ndkd1}
    \nu_{(X',E),\alpha}=\frac{1}{n}\cdot \mathrm{ord}_{F},\,\, \nu_{(X'_s,E_s),\alpha}=\frac{1}{n}\cdot \mathrm{ord}_{F_{s}}
    \end{equation}
    (cf. Lemma \ref{hel} (e)).
\end{itemize}

\smallskip

Let $s\in S'$, let $a>\sigma_{F_s}(D_s)$ be a rational number, and $h:X''_S\to X$ the composition of $X''_S\to X'_S$ and $X'_S\to X$. Denoted by $D:=-(K_X+\Delta)$. By the normality of $X''_s,X''$ and Lemma \ref{volume asymptotic order}, 
$$ \mathrm{vol}(h^*_sD_s-aF_{s})<\mathrm{vol}(h^*_sD_s)$$
holds for any $s\in S'$, and 
$$
\begin{aligned} 
\mathrm{vol}(h^*_{\overline{\eta}}D_{\overline{\eta}}-aF)&\le \mathrm{vol}(h^*_sD_s-aF_{s}) & (1)
\\ &<\mathrm{vol}(h^*_sD_s) &
\\ &=\mathrm{vol}(D_s) &
\\ &=\mathrm{vol}(D_{\overline{\eta}}) & (2)
\\ &=\mathrm{vol}(h^*_{\overline{\eta}}D_{\overline{\eta}}),&
\end{aligned}$$ 
where the semicontinuity of cohomology (cf. \cite{Kol23}*{Theorem 3.32.1}) is used in (1), and the constancy of anti-canonical volumes with Theorem \ref{main111} is applied in (2). Hence, $a>\sigma_{F}(D_{\overline{\eta}})$ by Lemma \ref{volume asymptotic order}. Thus, we have $$\sigma_{F_{s}}(D_s)\ge \sigma_{F}(D_{\overline{\eta}}) \text{ for any }s\in S'.$$
Consider the fact that the sets of divisorial valuations in $\mathrm{QM}(X',E)$ and $\mathrm{QM}(X'_s,E_s)$ form dense sets. By applying Equation (\ref{ndkd1}), and Lemma \ref{yae}, we obtain
\begin{equation}\label{main2}
\sigma_{\nu_{0s}}(D_s)\ge \sigma_{\nu_0}(D_{\overline{\eta}})\text{ for any }s\in S',
\end{equation}
where $\nu_{0s}:=\nu_{(X'_s,E_s),\alpha_0}$. Combining (\ref{main1}) with (\ref{main2}) results in $$
\begin{aligned}
\mathrm{lct}_{\sigma}(X_{\overline{\eta}},\Delta_{\overline{\eta}},D_{\overline{\eta}})&=\frac{A_{X_{\overline{\eta}},\Delta_{\overline{\eta}}}(\nu_0)}{\sigma_{\nu_0}(D_{\overline{\eta}})}
\\ &\ge \frac{A_{X_s,\Delta_s}(\nu_{0s})}{\sigma_{\nu_{0s}}(D_s)}
\\ &\ge \mathrm{lct}_{\sigma}(X_s,\Delta_s,D_s) & (3)
\end{aligned}$$ 
for any $s\in S'$, where we used \cite{Xu25}*{Lemma 1.60} in (3). Since $(X_s,\Delta_s)$ is of Fano type for any $s\in S'$, Lemma \ref{Fano type} ensures 
$$\mathrm{lct}_{\sigma}(X_s,\Delta_s,D_s)>1,$$
leading to $\mathrm{lct}_{\sigma}(X_{\overline{\eta}},\Delta_{\overline{\eta}},D_{\overline{\eta}})>1$. Thus, by Lemma \ref{Fano type}, $(X_{\overline{\eta}},\Delta_{\overline{\eta}})$ is of Fano type as desired.

\smallskip

Let us prove (b)$\implies$(a). Let us claim that after shrinking $S$ and replacing $S$ with a finite cover, we may assume that for any $s\in S$, $(X_s,\Delta_s)$ is of Fano type. Indeed, there is an effective $\Q$-divisor $\Delta'$ on $X_{\overline{\eta}}$ such that 
$$K_{X_{\overline{\eta}}}+\Delta_{\overline{\eta}}+\Delta'\sim_{\Q} 0.$$ 
Moreover, since $\Delta$ is determined by finitely many equations in $\overline{k(\eta)}$, we have that after shrinking $S$ and replacing $S$ with a finite cover, we may assume that there is an effective $\Q$-divisor $\Delta'_S$ such that $\Delta'_S\times_S \overline{\eta}=\Delta'$, and 
$$K_X+\Delta+\Delta'_S\sim_{\Q,S}0$$
(cf. Remark \ref{remark}). If we consider the log resolution of $f':X'\to (X_{\overline{\eta}},\Delta_{\overline{\eta}}+\Delta')$, after further shrinking $S$ and replacing $S$ with a finite cover, we can find a model $f'_S:X'_S\to (X,\Delta+\Delta'_S)$ of $f'$ over $S$ that is a log resolution of $(X,\Delta+\Delta'_S)$. Moreover, for every prime divisor $E$ on $X'_S$,
$$ A_{X,\Delta+\Delta'_S}(E)=A_{X_{\overline{\eta}},\Delta_{\overline{\eta}}+\Delta'}(E_{\overline{\eta}})>0,$$
and therefore we can deduce that $(X,\Delta+\Delta'_S)$ is klt (cf. Remark \ref{remark}). By \cite{Laz04b}*{Example 9.5.6 and Theorem 9.5.32}, for a general $s\in S$, $K_{X_s}+\Delta_s+\Delta'_S|_s\sim_{\Q} 0$ and $(X_s,\Delta_s+\Delta'_S|_s)$ is klt. Furthermore, since $-(K_{X_{\overline{\eta}}}+\Delta_{\overline{\eta}})$ is big, we obtain that there is an ample $\Q$-divisor $A$, and an effective $\Q$-divisor $E$ on $X_{\overline{\eta}}$ such that $-(K_{X_{\overline{\eta}}}+\Delta_{\overline{\eta}})=A+E$. After further shrinking $S$ and taking a finite cover of $S$, we may assume that there are $A_X,E_X$ on $X$ such that $-(K_X+\Delta)=A+E$, $A_X\times_S \overline{\eta}=A$, $E_X\times_S \overline{\eta}=E$, and $A_X$ is ample over $S$. Therefore, for a general $s\in S$, $-(K_{X_s}+\Delta_s)=A_X|_s+E_X|_s$, $A_X|_s$ is ample on $X_s$, $E_X|_s$ is effective on $X_s$, and hence $-(K_{X_s}+\Delta_s)$ is big on $X_s$. Thus, $(X_s,\Delta_s)$ is of Fano type.

We may assume that $X_{\overline{\eta}}$ is $\Q$-factorial after taking a $\Q$-factorialization, shrinking $S$, and taking a finite cover of $S$ (cf. \cite{Fuj11}*{Theorem 4.1}).

\smallskip

Let
$$ (X_0,\Delta_0):=(X_{\overline{\eta}},\Delta_{\overline{\eta}})\overset{f_1}{\dashrightarrow} (X_1,\Delta_1)\overset{f_2}{\dashrightarrow} \cdots \overset{f_n}{\dashrightarrow} (X_n,\Delta_n)$$
be a $-(K_{X_{\overline{\eta}}}+\Delta_{\overline{\eta}})$-MMP. Such an MMP exists because $(X_{\overline{\eta}},\Delta_{\overline{\eta}})$ is of Fano type, and any pair that is of Fano type has any $D$-MMP for an effective divisor $D$ (cf. \cite{BCHM10}*{Corollary 1.3.1}). Moreover, $(X_n,\Delta_n)$ is of Fano type by Lemma \ref{Fa1}, and thus we may consider the algebraic fiber space $X_n\to Y$ associated with a semiample divisor $-(K_{X_n}+\Delta_n)$.

\smallskip

Let $H$ be an ample divisor on $Y$ such that $g^*H=-(K_{X_n}+\Delta_n)$. Since $f_i$ and $g$ are determined by finitely many equations in $\overline{k(\eta)}$, after shrinking $S$ and replacing $S$ with a finite cover, we may assume that there are $S$-maps $f_{iS}:X_{(i-1)S}\dashrightarrow X_{iS}$ and $g_S:X_{nS}\to Y_S$ such that $f_{iS}\times_S \overline{\eta}=f_i$ and $g_S\times_S \overline{\eta}=g$ (cf. Remark \ref{remark}). We may also assume that there is an ample divisor $H_S$ over $Y_S$ such that $H_S\times_S \overline{\eta}=H$. Furthermore, we may assume that $f_{iS}$ are $-(K_{X_{(i-1)S}}+\Delta_{i-1})$-negative over $S$. Indeed, let
$$
\begin{tikzcd}
& X'' \ar["\varphi"']{ld}\ar["\varphi'"]{rd}&\\
X_{i-1}\ar[dashed,"f_i"']{rr}& &X_i
\end{tikzcd}
$$
be a common resolution of $f_i$, and let $D_i:=-(K_{X_{iS}}+\Delta_{i})$. Then, since $f_i$ is $D_{i\overline{\eta}}$-negative,
$$ \varphi^*D_{(i-1)\overline{\eta}}=(\varphi')^*D_{i\overline{\eta}}+E'$$
for some effective $\varphi'$-exceptional divisor on $X''$ with $\mathrm{Exc}(\varphi')\subseteq \Supp E'$. After shrinking $S$ and replacing $S$ with a finite cover, we may assume that there exist models $\varphi_S:X''_S\to X_{(i-1)S}$, $\varphi'_S:X''_S\to X_{iS}$, and $E'_S$ of $\varphi,\varphi',E'$ over $S$ with the commutative diagram
$$
\begin{tikzcd}
& X''_S \ar["\varphi_S"']{ld}\ar["\varphi'_S"]{rd}&\\
X_{(i-1)S}\ar[dashed,"f_{iS}"']{rr}& &X_{iS}
\end{tikzcd}
$$
that is a common resolution of $f_{iS}$, and the inclusion $\mathrm{Exc}(\varphi'_S)\subseteq \Supp E'_S$ (cf. Remark \ref{remark}). Moreover,
$$ \varphi^*_SD_{i-1}=(\varphi'_S)^*D_i+E'_S,$$
and therefore $f_{iS}$ is a $-(K_{X_{(i-1)S}}+\Delta_{i-1})$-negative map over $S$. Let $\Delta_{YS}:=g_{S*}\Delta_{nS}$, and let $n$ be a positive integer such that $nH_S$ is ample.

\smallskip

After shrinking $S$, we have that for any $s\in S$, $H_s:=H_S\times_Ss$ is an ample divisor on $Y_s:=Y_S\times_S s$. Since
\begin{equation} \label{hdd}
X_{0s}\overset{f_{1s}}{\dashrightarrow} X_{1s}\overset{f_{2s}}{\dashrightarrow} \cdots \overset{f_{ns}}{\dashrightarrow} X_{ns}
\end{equation}
is a sequence of $-(K_{X_{(i-1)s}}+\Delta_{(i-1)s})$-negative maps, where $X_{is}:=X_{iS}|_s$, $\Delta_{is}:=\Delta_{iS}|_s$ and $f_{is}:=f_{iS}|_s$, we have that $Y_s$ is of Fano type by Lemma \ref{Fa1} and \ref{Fa2}.

\smallskip

Therefore, by the Kodaira vanishing theorem, $H^1(Y_s,\mathcal{O}_{Y_s}(mnH_s))=0$ for any positive integer $m$. This implies that $R^1h_*\mathcal{O}_{Y_S}(mnH_S)=0$ by the upper semicontinuity of cohomology (cf. \cite{Har77}*{Theorem 12.8}). Thus, by \cite{Har77}*{Corollary 12.11}, $h_{S*}\mathcal{O}_{Y_S}(mnH_S)$ is torsion-free at $s$, and
$$ h_{S*}\mathcal{O}_{Y_S}(mnH_S)\otimes_{\mathcal{O}_S} k(s)=H^0(Y_s,\mathcal{O}_{Y_s}(mnH_s)).$$
for any positive integer $m$. Thus,
$$ \mathrm{vol}(H_s)=\limsup_{m\to \infty}\frac{\mathrm{rank}\,g_{S*}\mathcal{O}_{Y_S}(mH_S)}{\frac{m^{\dim X}}{(\dim X)!}}=\mathrm{vol}(H_{s'})$$
for all $s,s'\in S$. Moreover, by (\ref{hdd}), we obtain
$$ \mathrm{vol}(-(K_{X_s}+\Delta_s))=\mathrm{vol}(-(K_{X_{s'}}+\Delta_{s'}))$$
for all $s,s'\in S$.
\end{proof}

\begin{remark}
We anticipate that (\ref{hdd}) is a $-(K_{X_s}+\Delta_s)$-MMP, and this holds when $S$ is a variety over $\C$ (see \cite{CLZ25}*{Lemma 5.5}).
\end{remark}

\begin{proof}[Proof of Theorem \ref{surface}]
Note that the proof is almost a copy of \cite{GT16}*{Proof of Theorem 1.1}.

\smallskip

\noindent \textbf{Step 1}. Let us prove $X$ is $\Q$-Gorenstein. Assume that $X$ is smooth in codimension $1$. According to \cite{EGAIV2}*{9.9.2 (viii)}, the set
$$ E:=\{x\in X\mid \mathcal{O}_X|_{X_x}\text{ satisfies }(S_2)\} $$
is constructible and includes $S'$. This implies that the generic point $\eta$ of $S$ belongs to $E$, and consequently, after shrinking $S$, we can see that $X$ is normal. Additionally, for each $s\in S'$, $X_{s}$ has rational singularities and is $\Q$-factorial by \cite{KM98}*{Theorem 5.22} and the fact that any klt surface is $\Q$-factorial. This leads to the conclusion that $X$ is $\Q$-Gorenstein, as demonstrated by \cite{dFEM11}*{Theorem B.1}.

\noindent \textbf{Step 2}. We know that $X_{\overline{\eta}}$ is a surface over $\overline{k(\eta)}$, and thus we may consider the Zariski decomposition
$$ D:=-K_{X_{\overline{\eta}}}=P+N.$$
After shrinking $S$ and replacing $S$ with a finite cover, we may assume that there are $\R$-divisors $P_S,N_S$ on $X$ such that $P_S\times_S \overline{\eta}=P$, and $N_S\times_S \overline{\eta}=N$. By the generic flatness (cf. \cite{Stacks}*{Proposition 052A}), after further shrinking $S$, we may assume that both $P_S,N_S$ are flat over $S$.

\smallskip

\noindent \textbf{Step 3}. Let $N=\sum a_iC_i$ be the decomposition of $N_S$ into prime divisors. After shrinking $S$ and replacing $S$ with a finite cover, we may assume that there are irreducible $C_{iS}\subseteq X$ such that $N_S=\sum a_iC_{iS}$ and $C_{iS}\times_S \overline{\eta}=C_i$ (cf. Remark \ref{remark}). Let us denote by $P_s:=P_S|_s, N_s:=N_S|_s$, and $C_{is}:=C_{iS}|_s$ for any $s\in S$. Suppose $N(s)$ is the negative part of the Zariski decomposition of $D_s$. Our claim is
$$ N_s\le N(s).$$

\smallskip

For any general closed point $s\in S$,
$$ P_s\cdot C_{is}=P_S\cdot C_{iS}=P\cdot C_i=0\text{ for all }i.$$
If $P_s$ is nef, then $D_s=P_s+N_s$ is the Zariski decomposition of $D_s$, and thus the claim holds by the uniqueness of the Zariski decomposition (cf. \cite{Fuj79}*{(1.12) Theorem}). Let us assume that $P_s$ is not nef.

\smallskip

Note that $P_s$ is pseudo-effective (cf. \cite{Fuj79}*{(1.8) Lemma}), and therefore $P_s$ has the Zariski decomposition $P_s=P'+N'$. Let $N'':=N_s+N'$, and $N''=\sum a'_iC'_i$ be the decomposition of $N''$ into prime divisors. It is enough to show that $D_s=P'+N''$ is the Zariski decomposition of $D_s$.

\smallskip

Let us assume that $P'\cdot C_{is}\ne 0$. Then $N''\cdot C_{is}<0$, and therefore $C_{is}$ is contained in the support of $N''$, and this contradicts the fact that $P_s=P'+N'$ is the Zariski decomposition. Hence, $P\cdot C_{is}=0$, and $C_{is}\cdot C'_j=0$ for all $i,j$. Thus, $N$ has a negative definite intersection matrix, and by the uniqueness of the Zariski decomposition, $D_s=P'+N''$ is the Zariski decomposition of $D_s$. Hence, the claim follows.

\smallskip

\noindent \textbf{Step 4}. Next, by the definition of Fano type pair, we see that for any $s\in S'$, there is an effective $\Q$-divisor $\Delta(s)$ on $X_s$ such that $(X_s,\Delta(s))$ is klt, and $K_{X_s}+\Delta(s)\sim_{\Q}0$. By the property of the Zariski decomposition (cf. \cite{Nak04}*{1.14 Proposition (2) in Chapter 3}),
$$ \Delta(s)\ge N(s)\ge N_s.$$
Therefore, $(X_s,N_s)$ is a Fano type pair for all $s\in S'$, and therefore $(X_{\overline{\eta}},N)$ is klt by \cite{Laz04b}*{Theorem 9.5.16}.

\smallskip

\noindent \textbf{Step 5}. Let us show $P_s$ is nef for all $s\in S'$. By \cite{GT16}*{Lemma 4.4}, we obtain there exists an effective $\Q$-divisor $D=\sum b_iD_i$ with irreducible $D_i$ on $X_{\overline{\eta}}$ such that $P\sim_{\Q} D$. If $P_s$ is not nef, then there exists some irreducible curve $C$ on $X_s$ such that $P_s\cdot C<0$. Then $C\subseteq \Supp D_s$, that is, there exists $i$ such that $C=D_{is}$. Note that after shrinking $S$ and replacing $S$ with a finite cover, we may assume $D_{is}$ are irreducible. On the other hand,
$$ P_s\cdot C=P_s\cdot D_{is}=P\cdot D_i\ge 0,$$
and this is a contradiction. Hence, $P_s$ is nef for all $s\in S'$. Note that $P_s\cdot N_{js}=P\cdot N_j=0$ for every irreducible component $N_j$ of $N$, and thus $N_s$ is zero or has a negative definite intersection matrix. Hence, by the uniqueness of the Zariski decomposition,
$$ -K_{X_s}=P_s+N_s$$
is the Zariski decomposition.

\smallskip

Observe that because $-K_{X_s}$ is big, $P_s$ is also big, and therefore $P^2=P^2_s\ge 0$, confirming that $P$ is big follows. Thus,
$$ -(K_{X_{\overline{\eta}}}+N)=P$$
is big and nef. We proved in Step 3 that $(X_{\overline{\eta}},N)$ is klt, hence establishing that $X_{\overline{\eta}}$ is of Fano type.
\end{proof}

\begin{proof}[Proof of Corollary \ref{coro1}]
Let us consider the set $\{X_i\}_{i\in \Z_{>0}}$, which is a bounded set of surfaces of Fano type, where the sequence $\{\mathrm{vol}(-K_{X_i})\}_{i\in \Z_{>0}}$ decreases strictly. By the definition of boundedness, there exists a projective morphism $X\to S$ of varieties, and for every $i$, a closed point $s_i\in S$ is present such that $X_i=X_{s_i}$. We can assume that the set $\{s_i\}$ is dense in $S$.

\smallskip

As in the proof of Theorem \ref{surface}, we can prove $X$ is $\Q$-Gorenstein, and we see $X_{\overline{\eta}}$ is of Fano type by Theorem \ref{surface}. Consequently, applying (b)$\implies$(a) in Theorem \ref{imain1}, it is deduced that there exists a constant $v>0$ such that for infinitely many $i$, $\mathrm{vol}(-K_{X_i})=v$. This contradicts the assertion that $\mathrm{vol}(-K_{X_i})$ is strictly decreasing.
\end{proof}

\begin{example}\label{ex}
Let us explain why the boundedness of $\mathcal{S}$ in Corollary \ref{coro1} is required. Consider a positive integer $d$ and a DCC subset $I\subseteq [0,1]$. According to \cite{HMX14}*{Theorem 1.3}, the set
$$ \left\{\mathrm{vol}(K_X+\Delta)\mid(X,\Delta)\text{ is lc, }\dim X=d,\text{ with }\Delta\text{ coefficients in }I\right\}$$
is a DCC set, and there is no assumption of boundedness at all. This suggests that Corollary \ref{coro1} holds without the condition that $\mathcal{S}$ is bounded. Specifically, one might propose the following.

\smallskip

\noindent \textit{Question}. Given a DCC subset $I\subseteq [0,1]$, is it true that the set
$$ \mathcal{S}\subseteq \left\{\mathrm{vol}(-(K_X+\Delta))\mid(X,\Delta)\text{ is of klt Fano type, }\dim X=d,\text{ with }\Delta\text{ coefficients in }I\right\}$$
is also a DCC set?

\smallskip

Unfortunately, the answer is negative even when $I=0$, as demonstrated by a counterexample in \cite{HMX14}*{Example 2.1.1}. The example shows that even if $\mathcal{S}$ is a birationally bounded family, the question has a negative answer. An additional counterexample provided in \cite{CLZ25}*{Remark 6.1} further illustrates that the question is incorrect even if there is a variety $X$ such that for any $X'\in \mathcal{S}$, there is a proper birational morphism $f:X'\to X$.
\smallskip

Define $X:=\P^2$, and consider two divisors $\Delta, \Delta' \in |\mathcal{O}_X(1)|$ such that $\Delta+\Delta'$ is a simple normal crossing (snc) divisor. Let $f_1: X'_1 \to X$ be the blowup along $\Delta\cap\Delta'$, and $E_1$ be the exceptional divisor. Set $c_n:=1-\frac{1}{n+1}$, $g_1:=f_1$, and let $\Delta_1,\Delta'_1$ be the strict transforms of $\Delta$ and $\Delta'$ on $X'_1$ as $\Delta_1$ and $\Delta'_1$, respectively. We then have the following equality:
$$ K_{X'_1}+c_n\Delta_1+c_n\Delta'_1+\left(1-\frac{2}{n+1}\right)E_1=g^*_1(K_X+c_n\Delta +c_n \Delta').$$
Subsequently, consider the blowup $f_2:X'_2\to X'_1$ along $\Delta_1\cap E_1$, let $E_2$ be the exceptional divisor, and denote the strict transforms of $\Delta$ and $\Delta'$ on $X'_2$ by $\Delta_2$ and $\Delta'_2$. Let $g_2:=f_1\circ f_2$. We have:
$$ K_{X'_2}+c_n\Delta_2+c_n\Delta'_2+\left(1-\frac{2}{n+1}\right)E_1+\left(1-\frac{3}{n+1}\right)E_2=g^*_2(K_X+c_n\Delta+c_n\Delta'). $$
Proceeding inductively, we construct a proper birational morphism $g_n:X'_n\to X$, with prime divisors $E_1,\cdots, E_n$ on $X'_n$ and strict transforms $\Delta_n,\Delta'_n$ of $\Delta,\Delta'$ on $X'_n$, satisfying:
$$ K_{X'_n}+c_n\Delta_n+c_n\Delta'_n+\sum_{i=1}^{n}\left(1-\frac{i+1}{n+1}\right)E_i=g^*_n(K_X+c_n\Delta+c_n \Delta'). $$
By contracting $E_1, \cdots, E_{n-1}$, a contraction $h'_n: X'_n \to X_n$ is obtained over $X$, and this is possible by \cite{BCHM10}*{Theorem 1.2 (1)}. Let $h_n:X_n\to X$ be the corresponding morphism with $\Delta_{X_n},\Delta'_{X_n}$ as the strict transforms of $\Delta_n,\Delta'_n$ on $X_n$. In particular, $h_n:X_n\to X$ only extracts the prime divisor $E_n$ out of those over $X$. By construction, this leads to:
$$ K_{X_n}+c_n\Delta_{X_n}+c_n\Delta'_{X_n}=h^*_n(K_X+c_n\Delta+c_n\Delta'). $$
This shows that $(X_n,c_n\Delta_{X_n}+c_n\Delta'_{X_n})$, and thus $(X_n,0)$ is a klt Fano type pair. Let $\mathcal{S}:=\{X_n\}_{n\in \Z_{>0}}$.

\smallskip

We claim that the set $\{\mathrm{vol}(-K_{X_n})\}_{n\in \Z_{>0}}$ can be a counterexample. We observe that
$$ K_{X'_n}+\Delta_n=g^*_n(K_X+\Delta),$$
implying that $K_{X_n}+\Delta_{X_n}=h^*_n(K_X+\Delta)$. Moreover, it is worth noting that $\Delta_{n}\cdot E_i=0$ for each $1\le i\le n-1$, which makes $\Delta_{n}$ $h'_n$-trivial. Therefore, by the cone theorem (cf. \cite{KM98}*{Theorem 3.7 (4)}), we have $\Delta_n=h'^*_n\Delta_{X_n}$, and thus
$$ \mathrm{vol}(-K_{X_n})=\mathrm{vol}(-h^*_n(K_X+\Delta)+\Delta_{X_n})=\mathrm{vol}(-g^*_n(K_X+\Delta)+\Delta_n).$$
Let us prove that the sequence $a_n:=\mathrm{vol}(-g^*_n(K_X+\Delta)+\Delta_n)$ is strictly decreasing, implying that $\{a_n\}$ cannot be a DCC set. Indeed,
$$ 
\begin{aligned}
a_1&=\mathrm{vol}(-g^*_1(K_X+\Delta)+\Delta_1)
\\ &=\mathrm{vol}(-g^*_2(K_X+\Delta)+f^*_2\Delta_1)
\\ &>a_2=\mathrm{vol}(-g^*_2(K_X+\Delta)+\Delta_2) & (1)
\\ &=\mathrm{vol}(-g^*_3(K_X+\Delta)+f^*_3\Delta_2)
\\ &>a_3=\mathrm{vol}(-g^*_3(K_X+\Delta)+\Delta_3) & (2)
\\ & =\cdots,
\end{aligned}
$$
where Lemma \ref{volume asymptotic order} is used in (1) and (2), and therefore $a_1>a_2>a_3>\cdots $ and we have the assertion. Consequently, a positive answer to the question is not assured, even if there is a variety $X$ such that for any $X'\in \mathcal{S}$, there is a proper birational morphism $f:X'\to X$.

Note that \cite{CLZ25}*{Remark 6.1} proves that the set $\mathcal{S}$ does not form a bounded family.
\end{example}

\smallskip

Let us prove Theorem \ref{imain2}. Observe that the proof follows the same structure as that of Theorem \ref{imain1}.

\begin{proof}[Proof of Theorem \ref{imain2}]
Let us prove (a)$\implies$(b); By Theorem \ref{main222}, $-(K_X+\Delta)$ is big, and $(X,\Delta)$ is klt by \cite{HW02}*{Theorem}. Theorem \ref{Xuapp} gives a log smooth model $(X',E)$ over $X$ and a quasi-monomial valuation $\nu_0\in \mathrm{QM}\left(X',E=\sum E_i\right)$ such that $\nu_0$ computes $\mathrm{lct}_{\sigma}(X,\Delta,-(K_{X}+\Delta))$. Let $E_1,\cdots,E_r\subseteq E$ be prime divisors that define $\nu_0$, and $\alpha_0\in \R^r_{\ge 0}$ a tuple satisfying $\nu_0=\nu_{(X',E),\alpha_0}$.

\smallskip

After shrinking $R$, there is a model $(X'_R,E_{1R},\cdots,E_{rR},E_R)$ of $(X',E_1,\cdots,E_r,E)$ over $R$ (cf. Remark \ref{remark}). Moreover, we may assume
\begin{itemize}
\item $R$ is regular,
    \item $E_R$ has simple normal crossing,
    \item $X_R\to \Spec R$ and $X'_R\to \Spec R$ are flat with reduced and irreducible fibers, and
    \item $X'_R\to X_R$ is a fiberwise birational morphism over $R$.
\end{itemize}

Additionally, by Proposition \ref{disc 2}, after shrinking $R$, we have \begin{equation} \label{main11}
A_{X,\Delta}(\nu_{(X',E),\alpha_0})=A_{X_{\overline{s}},\Delta_{\overline{s}}}(\nu_{(X'_{\overline{s}},E_{\overline{s}}),\alpha_0})
\end{equation}
for any $s\in \Spec R$.

\smallskip

Let $\alpha:=(\alpha_1,\cdots,\alpha_r)\in \Q^r_{\ge 0}$, and let $n$ be the least common multiple of the denominators of $\alpha_1,\cdots,\alpha_r$. Take a weighted blow-up $X''\to X'_R$ of weight $(n\cdot \alpha_1,\cdots,n\cdot \alpha_r)$, $X''_{\overline{s}}:=X''_R|_{\overline{s}}$ and $X'':=X''_R\times_R \C$. Then both $X''_{\overline{s}}$ and $X''$ are weighted blowups of $(X'_{\overline{s}},E_{\overline{s}})$ and $(X',E)$ of weight $(n\cdot \alpha_1,\cdots,n\cdot \alpha_r)$ respectively (cf. Lemma \ref{hel} (d))

\smallskip

The varieties and morphisms can be illustrated by the following diagram:
$$
\begin{tikzcd}
X''_{\overline{s}}\ar[hook]{d}\ar["g_{\overline{s}}"]{r}& X'_{\overline{s}} \ar[hook]{d}\ar{r}& X_{\overline{s}}\ar[hook]{d} \\
X''_R \ar{r}& X'_R \ar{r}& X_R \\
X''\ar{u}\ar["g"]{r}& X' \ar{u}\ar{r}& X\ar{u}
\end{tikzcd}
$$
The vertical morphisms are fiber products, and the horizontal morphisms are birational morphisms.
\smallskip

Let $s\in S'$, and let us note the facts we will use:

\begin{itemize}
    \item[(a)] $X''$ and $X''_{\overline{s}}:=X''_R|_{\overline{s}}$ are normal (cf. Lemma \ref{hel} (c)), and
    \item[(b)] if $F$ and $F_{\overline{s}}$ are toroidal divisors extracted by $g$ and $g_{\overline{s}}:=g_R|_{\overline{s}}$, then
    \begin{equation} \label{ndkd2}\nu_{(X',E),\alpha}=\frac{1}{n}\cdot \mathrm{ord}_{F},\,\, \nu_{(X'_{\overline{s}},E_{\overline{s}}),\alpha}=\frac{1}{n}\cdot \mathrm{ord}_{F_{\overline{s}}}
    \end{equation}
    (cf. Lemma \ref{hel} (e)).
\end{itemize}

\smallskip

Let us denote by $D:=-(K_{X_R}+\Delta_R)$, and $D_{\C}:=D\times_R \C=-(K_X+\Delta)$. Let $a>\sigma_{F_{{\overline{s}}}}(D_{\overline{s}})$ be any rational number, $h:X''\to X_R$ the composition of $X''\to X'_R$ and $X'_R\to X_R$, and let $h_{\C}:=h\times_R \C$. Then by the normality of $X''_{\overline{s}},X''$ and Lemma \ref{volume asymptotic order},
$$ \mathrm{vol}(h^*_{\overline{s}}D_{\overline{s}}-aF_{{\overline{s}}})<\mathrm{vol}(h^*_{\overline{s}}D_{\overline{s}}).$$
Therefore, for any $s \in S'$, we have:
$$
\begin{aligned}
\mathrm{vol}(h^*_{\C}D_{\C}-aF) &\leq \mathrm{vol}(h^*_{\overline{s}}D_{\overline{s}}-aF_{\overline{s}}) & (1)
\\ &< \mathrm{vol}(h^*_{\overline{s}}D_{\overline{s}})
\\ &= \mathrm{vol}(D_{\overline{s}})
\\ &= \mathrm{vol}(D_{\C}) & (2)
\\ &= \mathrm{vol}(h^*_{\C}D_{\C}),
\end{aligned}$$
where the semicontinuity of cohomology (cf. \cite{Kol23}*{Theorem 3.32.1}) is used in (1), and the constancy of anti-canonical volumes with Theorem \ref{main222} is applied in (2). Therefore,
$$ \mathrm{vol}(h^*_{\C}D_{\C}-aF)<\mathrm{vol}(h^*_{\C}D_{\C}),$$
and hence $a>\sigma_{F}(D_{\C})$ by Lemma \ref{volume asymptotic order}. Thus,
\begin{equation} \label{main22} 
\sigma_{F_{\overline{s}}}(D_{\overline{s}})\ge \sigma_{F}(D_{\C}) \text{ for any }s\in S'.\end{equation}
Combining (\ref{main11}) with (\ref{main22}) gives us
\begin{equation} \label{bound}
\frac{A_{X,\Delta}(F)}{\sigma_{F}(D_{\C})} \ge \frac{A_{X_{\overline{s}},\Delta_{\overline{s}}}(F_{\overline{s}})}{\sigma_{F_{\overline{s}}}(D_{\overline{s}})}\text{ for any }s\in S'.
\end{equation}
Since $(X_{\overline{s}},\Delta_{\overline{s}})$ is globally $F$-regular, by Theorem \ref{x} we have that there is an effective $\Q$-divisor $\Delta'$ such that
\begin{itemize}
    \item $(X_{\overline{s}},\Delta_{\overline{s}}+\Delta')$ is globally $F$-regular, and
    \item $K_{X_{\overline{s}}}+\Delta_{\overline{s}}+\Delta'\sim_{\Q}0$.
\end{itemize}
Moreover, by \cite{BSTZ10}*{Main Theorem}, and Lemma \ref{cri}, we obtain that there is a rational number $\varepsilon_0>0$ such that $\tau(X_{\overline{s}},\Delta_{\overline{s}}+(1+\varepsilon_0)\Delta')=\mathcal{O}_{X_{\overline{s}}}$ (note that $K_{X_{\overline{s}}}+\Delta_{\overline{s}}$ is $\Q$-Cartier, and let $\mathfrak{a}$ be the ideal sheaf of $n\Delta'$ for a sufficiently large positive integer $n$). Hence,
$$ \tau\left(X_{\overline{s}},\Delta_{\overline{s}},\frac{1}{m}|m(1+\varepsilon_0)D_{\overline{s}}|\right)=\mathcal{O}_{X_{\overline{s}}}$$
for any sufficiently divisible $m$.

\smallskip

Let $f'':X'''\to X_{\overline{s}}$ be a proper birational morphism, and $F'$ be a divisor on $X'''$ such that 
$$\mathfrak{b}(|m(1+\varepsilon_0)D_{\overline{s}}|)\cdot \mathcal{O}_{X'''}=\mathcal{O}_{X'''}(-F').$$ Using Theorem \ref{test to multiplier}, we get $$(f'')_*\mathcal{O}_{X'''}\left(\ceil{K_{X'''}-(f'')^*(K_{X_{\overline{s}}}+\Delta_{\overline{s}})-\frac{1}{m}F'}\right)=\mathcal{O}_{X_{\overline{s}}}$$
and $$ \frac{A_{X_{\overline{s}},\Delta_{\overline{s}}}(F_{\overline{s}})}{\frac{1}{m}\mathrm{ord}_{F_{\overline{s}}}|m(1+\varepsilon_0)D_{\overline{s}}|}>1,$$ leading to $$\frac{A_{X_{\overline{s}},\Delta_{\overline{s}}}(F_{\overline{s}})}{\sigma_{F_{\overline{s}}}(D_{\overline{s}})}\ge 1+\varepsilon_0.$$ Since the sets of divisorial valuations in $\mathrm{QM}(X',E)$ and $\mathrm{QM}(X'_{\overline{s}},E_{\overline{s}})$ are dense, (\ref{ndkd2}, \ref{bound}), and Lemma \ref{yae} show $$ \mathrm{lct}_{\sigma}(X,\Delta,D_{\C})=\frac{A_{X,\Delta}(\nu_0)}{\sigma_{\nu_0}(D_{\C})}>1,$$ proving $(X,\Delta)$ is of Fano type by Lemma \ref{Fano type}.

\smallskip

Let us prove (b)$\implies$(a). Note that after shrinking $R$, we may assume that $X_{\overline{s}}$ is globally $F$-regular for all $s\in \Spec R$ (cf. \cite{SS10}*{Theorem 1.2}). Furthermore, by further shrinking $R$ even more, we can assume that $X$ is $\Q$-factorial (see \cite{GLPSTZ15}*{Proposition 3.3}).

\smallskip

Let
$$ (X_0,\Delta_0):=(X,\Delta)\overset{f_1}{\dashrightarrow} (X_1,\Delta_1)\overset{f_2}{\dashrightarrow} \cdots \overset{f_n}{\dashrightarrow} (X_n,\Delta_n)$$
be a $-(K_{X}+\Delta)$-MMP. Such an MMP exists because $(X,\Delta)$ is of Fano type, and any pair that is of Fano type has any $D$-MMP for an effective divisor $D$ (cf. \cite{BCHM10}*{Corollary 1.3.1}). Moreover, $(X_n,\Delta_n)$ is of Fano type by Lemma \ref{Fa1}, and thus we may consider the algebraic fiber space $X_n\to Y$ associated to a semiample divisor $-(K_{X_n}+\Delta_n)$.

\smallskip

Let $H$ be an ample divisor $Y$ such that $g^*H=-(K_{X_n}+\Delta_n)$. After shrinking $R$, we may assume that there are $R$-maps $f_{iR}:X_{(i-1)R}\dashrightarrow X_{iR}$ and $g_R:X_{nR}\to Y_R$ such that $f_{iR}\times_R\C=f_i$ and $g_R\times_R \C=g$. We may further assume that there is an ample divisor $H_R$ over $Y_R$ such that $H_R\times_R \C=H$. Furthermore, we may assume that $f_{iR}$ are $-(K_{X_{(i-1)R}}+\Delta_{i-1})$-negative maps (which may not be elementary contractions). Let $\Delta_{YR}:=g_{R*}\Delta_{nR}$, and let $n$ be a positive integer such that $nH_R$ is Cartier.

\smallskip

Shrinking $R$, we have that for every $s\in \Spec R$, the divisor $H_{\overline{s}}:=H_R\times_R \overline{s}$ is ample on $Y_{\overline{s}}:=Y_R\times_R \overline{s}$. Note the sequence:
\begin{equation} \label{ddd}X_{0\overline{s}}\overset{f_{1\overline{s}}}{\dashrightarrow} X_{1\overline{s}}\overset{f_{2\overline{s}}}{\dashrightarrow} \cdots \overset{f_{n\overline{s}}}{\dashrightarrow} X_{n\overline{s}}.
\end{equation}
That is a sequence of maps that are $-(K_{X_{\overline{s}}}+\Delta_{\overline{s}})$-negative, with $X_{i\overline{s}}:=X_{iR}|_{\overline{s}}$ and $f_{i\overline{s}}:=f_{iR}|_{\overline{s}}$. Using (\ref{ddd}) and \cite{SS10}*{Proposition 6.3}, we can prove that $Y_{\overline{s}}$ is a globally $F$-regular variety.

\smallskip

Therefore, by \cite{Smi00}*{Corollary 4.3}, $H^1(Y_{\overline{s}},\mathcal{O}_{Y_{\overline{s}}}(mnH_{\overline{s}}))=0$ for any positive integer $m$. Thus, by the flat base change theorem (cf. \cite{Stacks}*{Lemma 02KH}) and the upper semicontinuity of cohomology (cf. \cite{Har77}*{Theorem 12.8}), $H^1(Y_R,\mathcal{O}_{Y_R}(mnH_R))=0$. Then \cite{Har77}*{Corollary 12.11} gives us that $h_{R*}\mathcal{O}_{Y_R}(mnH_R)$ is torsion-free at $s$, and
$$ h_{R*}\mathcal{O}_{Y_R}(mnH_R)\otimes_R k(s)=H^0(Y_s,\mathcal{O}_{Y_s}(mnH_s))$$
for any positive integer $m$. Thus, by the flat base change theorem (cf. \cite{Stacks}*{Lemma 02KH}) again,
$$ \mathrm{vol}(H_{\overline{s}})=\limsup_{m\to \infty}\frac{\mathrm{rank}\,g_{R*}\mathcal{O}_{Y_R}(mH_R)}{\frac{m^{\dim X}}{(\dim X)!}}=\mathrm{vol}(H_{\overline{s'}})$$
for all $s,s'\in \Spec R$. Moreover, by the fact that the composition of (\ref{ddd}) is a $-(K_{X_{\overline{s}}}+\Delta_{\overline{s}})$-negative map, we obtain
$$ \mathrm{vol}(-(K_{X_{\overline{s}}}+\Delta_{\overline{s}}))=\mathrm{vol}(-(K_{X_{\overline{s'}}}+\Delta_{\overline{s'}}))$$
for all $s,s'\in S$.
\end{proof}


\begin{thebibliography}{}

\bib{Bir19}{article}{
      author={Birkar, Caucher},
       title={Anti-pluricanonical systems on {F}ano varieties},
        date={2019},
        ISSN={0003-486X,1939-8980},
     journal={Ann. of Math. (2)},
      volume={190},
      number={2},
       pages={345\ndash 463},
         url={https://doi-org-ssl.access.yonsei.ac.kr/10.4007/annals.2019.190.2.1},
      review={\MR{3997127}},
}

\bib{Bir24}{article}{
      author={Birkar, Caucher},
       title={Boundedness of {F}ano type fibrations},
        date={2024},
        ISSN={0012-9593,1873-2151},
     journal={Ann. Sci. \'Ec. Norm. Sup\'er. (4)},
      volume={57},
      number={3},
       pages={787\ndash 840},
      review={\MR{4773297}},
}

\bib{BCHM10}{article}{
      author={Birkar, Caucher},
      author={Cascini, Paolo},
      author={Hacon, Christopher~D.},
      author={McKernan, James},
       title={Existence of minimal models for varieties of log general type},
        date={2010},
        ISSN={0894-0347,1088-6834},
     journal={J. Amer. Math. Soc.},
      volume={23},
      number={2},
       pages={405\ndash 468},
         url={https://doi-org-ssl.access.yonsei.ac.kr/10.1090/S0894-0347-09-00649-3},
      review={\MR{2601039}},
}

\bib{BLX22}{article}{
    AUTHOR = {Blum, Harold},
    author = {Liu, Yuchen},
    author = {Xu, Chenyang},
     TITLE = {Openness of {K}-semistability for {F}ano varieties},
   JOURNAL = {Duke Math. J.},
    VOLUME = {171},
      YEAR = {2022},
    NUMBER = {13},
     PAGES = {2753--2797},
      ISSN = {0012-7094,1547-7398},
  review = {\MR{4505846}},
}

\bib{BSTZ10}{article}{
      author={Blickle, Manuel},
      author={Schwede, Karl},
      author={Takagi, Shunsuke},
      author={Zhang, Wenliang},
       title={Discreteness and rationality of {$F$}-jumping numbers on singular varieties},
        date={2010},
        ISSN={0025-5831,1432-1807},
     journal={Math. Ann.},
      volume={347},
      number={4},
       pages={917\ndash 949},
         url={https://doi-org-ssl.access.yonsei.ac.kr/10.1007/s00208-009-0461-2},
      review={\MR{2658149}},
}

\bib{BPDD13}{article}{
      author={Boucksom, S\'ebastien},
      author={Demailly, Jean-Pierre},
      author={Păun, Mihai},
      author={Peternell, Thomas},
       title={The pseudo-effective cone of a compact {K}\"ahler manifold and varieties of negative {K}odaira dimension},
        date={2013},
        ISSN={1056-3911,1534-7486},
     journal={J. Algebraic Geom.},
      volume={22},
      number={2},
       pages={201\ndash 248},
         url={https://doi-org-ssl.access.yonsei.ac.kr/10.1090/S1056-3911-2012-00574-8},
      review={\MR{3019449}},
}

\bib{Can20}{article}{
      author={Canton, Eric},
       title={Berkovich log discrepancies in positive characteristic},
        date={2020},
        ISSN={1558-8599,1558-8602},
     journal={Pure Appl. Math. Q.},
      volume={16},
      number={5},
       pages={1465\ndash 1532},
         url={https://doi-org-ssl.access.yonsei.ac.kr/10.4310/PAMQ.2020.v16.n5.a5},
      review={\MR{4221003}},
}

\bib{CKT24}{article}{
      author={Cascini, Paolo},
      author={Kawakami, Tatsuro},
      author={Takagi, Shunsuke},
       title={Threefolds of globally {F}-regular type with nef anti-canonical divisor},
        date={2024},
     journal={arXiv preprint arXiv:2410.03871},
}

\bib{CLZ25}{article}{
      author={Choi, Sung~Rak},
      author={Li, Zhan},
      author={Zhou, Chuyu},
       title={Variation of cones of divisors in a family of varieties--{F}ano type case},
        date={2025},
     journal={arXiv preprint arXiv:2504.04109},
}

\bib{dFEM11}{incollection}{
      author={de~Fernex, Tommaso},
      author={Ein, Lawrence},
      author={Mustață, Mircea},
       title={Log canonical thresholds on varieties with bounded singularities},
        date={2011},
   booktitle={Classification of algebraic varieties},
      series={EMS Ser. Congr. Rep.},
   publisher={Eur. Math. Soc., Z\"urich},
       pages={221\ndash 257},
         url={https://doi.org/10.4171/007-1/10},
      review={\MR{2779474}},
}

\bib{Stacks}{article}{
      author={de~Jong~et al., A.~J.},
       title={The stacks project},
         url={http://stacks.math.columbia.edu},
}

\bib{Deb01}{book}{
      author={Debarre, Olivier},
       title={Higher-dimensional algebraic geometry},
      series={Universitext},
   publisher={Springer-Verlag, New York},
        date={2001},
        ISBN={0-387-95227-6},
         url={https://doi-org-ssl.access.yonsei.ac.kr/10.1007/978-1-4757-5406-3},
      review={\MR{1841091}},
}


\bib{Fuj11}{article}{
      author={Fujino, Osamu},
       title={Semi-stable minimal model program for varieties with trivial canonical divisor},
        date={2011},
        ISSN={0386-2194},
     journal={Proc. Japan Acad. Ser. A Math. Sci.},
      volume={87},
      number={3},
       pages={25\ndash 30},
         url={http://projecteuclid.org/euclid.pja/1299161391},
      review={\MR{2802603}},
}

\bib{Fuj79}{article}{
      author={Fujita, Takao},
       title={On {Z}ariski problem},
        date={1979},
        ISSN={0386-2194},
     journal={Proc. Japan Acad. Ser. A Math. Sci.},
      volume={55},
      number={3},
       pages={106\ndash 110},
         url={http://projecteuclid.org/euclid.pja/1195517399},
      review={\MR{531454}},
}

\bib{Fuj94}{article}{
      author={Fujita, Takao},
       title={Approximating {Z}ariski decomposition of big line bundles},
        date={1994},
        ISSN={0386-5991,1881-5472},
     journal={Kodai Math. J.},
      volume={17},
      number={1},
       pages={1\ndash 3},
         url={https://doi-org-ssl.access.yonsei.ac.kr/10.2996/kmj/1138039894},
      review={\MR{1262949}},
}

\bib{FKL16}{article}{
    author={Fulger, Mihai},
    author={Koll\'ar, J\'anos},
    author={Lehmann, Brian},
     title={Volume and {H}ilbert function of {$\mathbb{R}$}-divisors},
      date={2016},
      ISSN={0026-2285,1945-2365},
   journal={Michigan Math. J.},
    volume={65},
    number={2},
     pages={371--387},
      ISSN={0026-2285,1945-2365},
       url={https://doi.org/10.1307/mmj/1465329018},
    review={\MR{3510912}},
}

\bib{GLPSTZ15}{article}{
      author={Gongyo, Yoshinori},
      author={Li, Zhiyuan},
      author={Patakfalvi, Zsolt},
      author={Schwede, Karl},
      author={Tanaka, Hiromu},
      author={Zong, Runhong},
       title={On rational connectedness of globally {$F$}-regular threefolds},
        date={2015},
        ISSN={0001-8708,1090-2082},
     journal={Adv. Math.},
      volume={280},
       pages={47\ndash 78},
         url={https://doi.org/10.1016/j.aim.2015.04.012},
      review={\MR{3350212}},
}

\bib{GOST15}{article}{
      author={Gongyo, Yoshinori},
      author={Okawa, Shinnosuke},
      author={Sannai, Akiyoshi},
      author={Takagi, Shunsuke},
       title={Characterization of varieties of {F}ano type via singularities of {C}ox rings},
        date={2015},
        ISSN={1056-3911,1534-7486},
     journal={J. Algebraic Geom.},
      volume={24},
      number={1},
       pages={159\ndash 182},
         url={https://doi-org-ssl.access.yonsei.ac.kr/10.1090/S1056-3911-2014-00641-X},
      review={\MR{3275656}},
}

\bib{GT16}{article}{
      author={Gongyo, Yoshinori},
      author={Takagi, Shunsuke},
       title={Surfaces of globally {$F$}-regular and {$F$}-split type},
        date={2016},
        ISSN={0025-5831,1432-1807},
     journal={Math. Ann.},
      volume={364},
      number={3-4},
       pages={841\ndash 855},
         url={https://doi-org-ssl.access.yonsei.ac.kr/10.1007/s00208-015-1238-4},
      review={\MR{3466854}},
}

\bib{EGAIV2}{article}{
      author={Grothendieck, A.},
       title={\'El\'ements de G\'eom\'etrie Alg\'ebrique. {IV}. \'etude locale des sch\'emas et des morphismes de sch\'emas. {II}},
        date={1965},
        ISSN={0073-8301,1618-1913},
     journal={Inst. Hautes \'Etudes Sci. Publ. Math.},
      number={24},
       pages={231},
         url={http://www.numdam.org/item?id=PMIHES_1965__24__231_0},
      review={\MR{199181}},
}

\bib{FGAexplained}{book}{
author={Fantechi, Barbara},
author={G\"ottsche, Lothar},
author={Illusie, Luc},
author={Kleiman, Steven L.},
author={Nitsure, Nitin},
author={Vistoli, Angelo},
              title={Fundamental algebraic geometry},
              series={Mathematical Surveys and Monographs},
              publisher={American Mathematical Society, Providence, RI},
              date={2005},
              volume={Vol. 123},
              isbn={0-8218-3541-6},
              review={\MR{2222646}}
}

\bib{HMX13}{article}{
    AUTHOR = {Hacon, Christopher D.},
    author = {McKernan, James},
    author = {Xu, Chenyang},
     TITLE = {On the birational automorphisms of varieties of general type},
   JOURNAL = {Ann. of Math. (2)},
    VOLUME = {177},
      YEAR = {2013},
    NUMBER = {3},
     PAGES = {1077--1111},
  MRNUMBER = {\MR{3034294}},
}

\bib{HMX14}{article}{
      author={Hacon, Christopher~D.},
      author={McKernan, James},
      author={Xu, Chenyang},
       title={A{CC} for log canonical thresholds},
        date={2014},
        ISSN={0003-486X,1939-8980},
     journal={Ann. of Math. (2)},
      volume={180},
      number={2},
       pages={523\ndash 571},
         url={https://doi.org/10.4007/annals.2014.180.2.3},
      review={\MR{3224718}},
}

\bib{HMX18}{article}{
      author={Hacon, Christopher~D.},
      author={McKernan, James},
      author={Xu, Chenyang},
       title={Boundedness of moduli of varieties of general type},
        date={2018},
        ISSN={1435-9855,1435-9863},
     journal={J. Eur. Math. Soc. (JEMS)},
      volume={20},
      number={4},
       pages={865\ndash 901},
         url={https://doi.org/10.4171/JEMS/778},
      review={\MR{3779687}},
}

\bib{HW02}{article}{
      author={Hara, Nobuo},
      author={Watanabe, Kei-Ichi},
       title={{F}-regular and {F}-pure rings vs. log terminal and log canonical singularities},
        date={2002},
        ISSN={1056-3911,1534-7486},
     journal={J. Algebraic Geom.},
      volume={11},
      number={2},
       pages={363\ndash 392},
         url={https://doi.org/10.1090/S1056-3911-01-00306-X},
      review={\MR{1874118}},
}

\bib{Har77}{book}{
      author={Hartshorne, Robin},
       title={Algebraic geometry},
      series={Graduate Texts in Mathematics},
   publisher={Springer-Verlag, New York-Heidelberg},
        date={1977},
      volume={No. 52},
        ISBN={0-387-90244-9},
      review={\MR{463157}},
}

\bib{HH89}{incollection}{
      author={Hochster, Melvin},
      author={Huneke, Craig},
       title={Tight closure and strong {$F$}-regularity},
        date={1989},
       pages={119\ndash 133},
        note={Colloque en l'honneur de Pierre Samuel (Orsay, 1987)},
      review={\MR{1044348}},
}

\bib{HS06}{book}{
      author={Huneke, Craig},
      author={Swanson, Irena},
       title={Integral closure of ideals, rings, and modules},
      series={London Mathematical Society Lecture Note Series},
   publisher={Cambridge University Press, Cambridge},
        date={2006},
      volume={336},
        ISBN={978-0-521-68860-4; 0-521-68860-4},
      review={\MR{2266432}},
}

\bib{HP15}{article}{
      author={Hwang, DongSeon},
      author={Park, Jinhyung},
       title={Characterization of log del {P}ezzo pairs via anticanonical models},
        date={2015},
        ISSN={0025-5874,1432-1823},
     journal={Math. Z.},
      volume={280},
      number={1-2},
       pages={211\ndash 229},
         url={https://doi.org/10.1007/s00209-015-1419-6},
      review={\MR{3343904}},
}

\bib{Jia25}{article}{
      author={Jiao, Junpeng},
       title={The volume function is upper semicontinuous on families of divisors},
        date={2025},
     journal={arXiv preprint arXiv:2504.16676},
}

\bib{JM12}{article}{
      author={Jonsson, Mattias},
      author={Mustață, Mircea},
       title={Valuations and asymptotic invariants for sequences of ideals},
        date={2012},
        ISSN={0373-0956,1777-5310},
     journal={Ann. Inst. Fourier (Grenoble)},
      volume={62},
      number={6},
       pages={2145\ndash 2209},
      review={\MR{3060755}},
}


\bib{Kol23}{book}{
 author={Koll\'ar, J\'anos},
     title={Families of varieties of general type},
    series={Cambridge Tracts in Mathematics},
    volume={231},
      note={With the collaboration of Klaus Altmann and S\'andor J.
              Kov\'acs},
 publisher={Cambridge University Press, Cambridge},
      year={2023},
  review={\MR{4566297}},
}

\bib{KM98}{book}{
      author={Koll\'ar, J\'anos},
      author={Mori, Shigefumi},
       title={Birational geometry of algebraic varieties},
      series={Cambridge Tracts in Mathematics},
   publisher={Cambridge University Press, Cambridge},
        date={1998},
      volume={134},
        ISBN={0-521-63277-3},
         url={https://doi-org-ssl.access.yonsei.ac.kr/10.1017/CBO9780511662560},
        note={With the collaboration of C. H. Clemens and A. Corti, Translated from the 1998 Japanese original},
      review={\MR{1658959}},
}

\bib{Laz04a}{book}{
      author={Lazarsfeld, Robert},
       title={Positivity in algebraic geometry. {I}},
      series={Ergebnisse der Mathematik und ihrer Grenzgebiete. 3. Folge. A Series of Modern Surveys in Mathematics [Results in Mathematics and Related Areas. 3rd Series. A Series of Modern Surveys in Mathematics]},
   publisher={Springer-Verlag, Berlin},
        date={2004},
      volume={48},
        ISBN={3-540-22533-1},
         url={https://doi-org-ssl.access.yonsei.ac.kr/10.1007/978-3-642-18808-4},
        note={Classical setting: line bundles and linear series},
      review={\MR{2095471}},
}

\bib{Laz04b}{book}{
      author={Lazarsfeld, Robert},
       title={Positivity in algebraic geometry. {II}},
      series={Ergebnisse der Mathematik und ihrer Grenzgebiete. 3. Folge. A Series of Modern Surveys in Mathematics [Results in Mathematics and Related Areas. 3rd Series. A Series of Modern Surveys in Mathematics]},
   publisher={Springer-Verlag, Berlin},
        date={2004},
      volume={49},
        ISBN={3-540-22534-X},
         url={https://doi-org-ssl.access.yonsei.ac.kr/10.1007/978-3-642-18808-4},
        note={Positivity for vector bundles, and multiplier ideals},
      review={\MR{2095472}},
}

\bib{LX20}{article}{
author={Li, Chi},
author={Xu, Chenyang},
title={Stability of valuations and {K}oll\'ar components},
journal={J. Eur. Math. Soc. (JEMS)},
volume={22},
year={2020},
number={8},
pages={2573--2627},
issn={1435-9855,1435-9863},
review={\MR{4118616}},
}

\bib{Nak04}{book}{
      author={Nakayama, Noboru},
       title={Zariski-decomposition and abundance},
      series={MSJ Memoirs},
   publisher={Mathematical Society of Japan, Tokyo},
        date={2004},
      volume={14},
        ISBN={4-931469-31-0},
      review={\MR{2104208}},
}

\bib{Oka17}{article}{
      author={Okawa, Shinnosuke},
       title={Surfaces of globally {$F$}-regular type are of {F}ano type},
        date={2017},
        ISSN={0040-8735,2186-585X},
     journal={Tohoku Math. J. (2)},
      volume={69},
      number={1},
       pages={35\ndash 42},
         url={https://doi.org/10.2748/tmj/1493172126},
      review={\MR{3640012}},
}

\bib{Sch10}{article}{
      author={Schwede, Karl},
       title={Centers of {$F$}-purity},
        date={2010},
        ISSN={0025-5874,1432-1823},
     journal={Math. Z.},
      volume={265},
      number={3},
       pages={687\ndash 714},
         url={https://doi-org-ssl.access.yonsei.ac.kr/10.1007/s00209-009-0536-5},
      review={\MR{2644316}},
}

\bib{SS10}{article}{
      author={Schwede, Karl},
      author={Smith, Karen~E.},
       title={Globally {$F$}-regular and log {F}ano varieties},
        date={2010},
        ISSN={0001-8708,1090-2082},
     journal={Adv. Math.},
      volume={224},
      number={3},
       pages={863\ndash 894},
         url={https://doi-org-ssl.access.yonsei.ac.kr/10.1016/j.aim.2009.12.020},
      review={\MR{2628797}},
}

\bib{Smi00}{incollection}{
      author={Smith, Karen~E.},
       title={Globally {F}-regular varieties: applications to vanishing theorems for quotients of {F}ano varieties},
        date={2000},
      volume={48},
       pages={553\ndash 572},
         url={https://doi-org-ssl.access.yonsei.ac.kr/10.1307/mmj/1030132733},
        note={Dedicated to William Fulton on the occasion of his 60th birthday},
      review={\MR{1786505}},
}

\bib{SZ20}{article}{
      author={Sun, Xiaotao},
      author={Zhou, Mingshuo},
       title={Globally {$F$}-regular type of moduli spaces},
        date={2020},
        ISSN={0025-5831,1432-1807},
     journal={Math. Ann.},
      volume={378},
      number={3-4},
       pages={1245\ndash 1270},
         url={https://doi.org/10.1007/s00208-020-02077-3},
      review={\MR{4163526}},
}

\bib{Tak07}{article}{
      author={Takagi, Satoshi},
       title={Fujita's approximation theorem in positive characteristics},
        date={2007},
        ISSN={0023-608X},
     journal={J. Math. Kyoto Univ.},
      volume={47},
      number={1},
       pages={179\ndash 202},
         url={https://doi-org-ssl.access.yonsei.ac.kr/10.1215/kjm/1250281075},
      review={\MR{2359108}},
}

\bib{Tak04}{article}{
      author={Takagi, Shunsuke},
       title={An interpretation of multiplier ideals via tight closure},
        date={2004},
        ISSN={1056-3911,1534-7486},
     journal={J. Algebraic Geom.},
      volume={13},
      number={2},
       pages={393\ndash 415},
         url={https://doi-org-ssl.access.yonsei.ac.kr/10.1090/S1056-3911-03-00366-7},
      review={\MR{2047704}},
}

\bib{WW24}{article}{
      author={Wang, Jianping},
      author={Wen, Xueqing},
       title={Globally {F}-regular type of the moduli spaces of parabolic symplectic/orthogonal bundles on curves},
        date={2024},
        ISSN={2050-5094},
     journal={Forum Math. Sigma},
      volume={12},
       pages={Paper No. e69, 21},
         url={https://doi.org/10.1017/fms.2024.57},
      review={\MR{4749934}},
}

\bib{Xu20}{article}{
      author={Xu, Chenyang},
       title={A minimizing valuation is quasi-monomial},
        date={2020},
        ISSN={0003-486X,1939-8980},
     journal={Ann. of Math. (2)},
      volume={191},
      number={3},
       pages={1003\ndash 1030},
         url={https://doi-org-ssl.access.yonsei.ac.kr/10.4007/annals.2020.191.3.6},
      review={\MR{4088355}},
}

\bib{Xu23}{article}{
      author={Xu, Chenyang},
       title={K-stability for varieties with a big anticanonical class},
        date={2023},
        ISSN={2491-6765},
     journal={\'Epijournal G\'eom. Alg\'ebrique},
       pages={Art. 7, 9},
      review={\MR{4671736}},
}

\bib{Xu25}{book}{,
    author={Xu, Chenyang},
     title={K-stability of {F}ano varieties},
    series={New Mathematical Monographs},
    volume={50},
    publisher={Cambridge University Press, Cambridge},
      year={2025},
     pages={xi+411},
      isbn={978-1-009-53877-0; [9781009538763]},
    review={\MR{4893062}},
}

%\bib{Zar50}{article}{
%      author={Zariski, Oscar},
%       title={Sur la normalit\'e{} analytique des vari\'et\'es normales},
%        date={1950},
%        ISSN={0373-0956,1777-5310},
%     journal={Ann. Inst. Fourier (Grenoble)},
%      volume={2},
%       pages={161\ndash 164},
%         url={http://www.numdam.org/item?id=AIF_1950__2__161_0},
%      review={\MR{45413}},
%}

\end{thebibliography}
\end{document}